\documentclass[10pt,a4paper]{article}
\usepackage[utf8]{inputenc}
\usepackage{amsmath, amsthm}
\usepackage{amsfonts}
\usepackage{amssymb}
\usepackage{mathtools}
\usepackage{authblk}
\usepackage[longnamesfirst]{natbib}

\newtheorem{thm}{Theorem}[section]

\newtheorem{lem}{Lemma}[section]
\newtheorem{cor}{Corollary}[section]
\newtheorem{rem}{Remark}[section]
\newtheorem{assumption}{Assumption}[section]

\DeclarePairedDelimiter\floor{\lfloor}{\rfloor}
\newcommand{\Var}{\operatorname{Var}}
\newcommand{\Cov}{\operatorname{Cov}}
\newcommand{\I}{\text{I}}
\newcommand{\Transp}{^{\text{T}}}
\newcommand{\E}{\mathbb{E}}
\newcommand{\Prob}{\text{P}}

\begin{document}
\bibliographystyle{plainnat}
\title{\bf Functional Sieve Bootstrap for the Partial Sum Process with an Application to Change-Point Detection}

\date{\vspace{-1cm}}

\author[1]{Efstathios Paparoditis}
\author[2]{Lea Wegner}
\author[2]{Martin Wendler\thanks{The research was supported by German Research Foundation (Deutsche Forschungsgemeinschaft - DFG), project WE 5988/5 \emph{Graduelle Struktur\"anderungen in funktionalen Daten} and partially supported by a Cyprus Academy of Sciences, Letters and Arts research grant. We thank  Greg Rice for providing the R code for the method by \cite{AUE} we have used for comparison with our method, and Han Lin Shang for allowing us to use his R-Code for the functional sieve bootstrap.}\thanks{\texttt{ martin.wendler@ovgu.de}}}

\affil[1]{Cyprus Academy of Sciences, Letters and Arts}
\affil[2]{Otto-von-Guericke-University Magdeburg}
%\affil[3]{\texttt{\small martin.wendler@ovgu.de}}

\maketitle

{\bf Abstract}  This paper applies the functional  sieve bootstrap  (FSB)   to  estimate the distribution of   the partial sum process  for time series stemming from a  weakly stationary functional  process. Consistency of the FSB procedure under weak assumptions  on the underlying functional process  is established.  This   result allows for the application of  the FSB  procedure    to  testing for a change-point  in the mean of a    functional time series  using the CUSUM-statistic.
%, where the  latter   is a non-linear functional of the partial sum process. 
We   show   that the FSB    asymptotically correctly  estimates critical values of the CUSUM-based test under the null-hypothesis. Consistency of the  FSB-based test under local alternatives  also is proven. The finite sample performance of  the procedure is studied via simulations.

\medskip


\section{Introduction}

Consider a time series $ X_1, X_2, \ldots, X_n$ stemming from a stationary functional process $  {\mathcal X}:=\{X_s, s\in \mathbb{Z}\}$, where for each $s \in \mathbb{Z}$,  the random element $ X_s$,  takes values in a separable Hilbert space $ {\mathcal H}$ equipped with a inner product 
$ \langle\cdot,\cdot\rangle :{\mathcal H}\times {\mathcal H} \rightarrow  \mathbb{R}$ and the norm $ \|x\| := \langle x,x\rangle^{1/2}$,  $ x\in{\mathcal H}$.  
% In this paper we assume that 
%$ {\mathcal H}=L^2([0,1])$, that is the space of square integrable functions defined on the interval $[0,1]$. 
For any  $t \in [0,1]$ consider the random process  $  Z_n=(Z_n(t) )_{ t\in[0,1]}$, where $Z_n(t)$ is  defined as the partial sum, 
\begin{equation}  \label{eq.DefPS}
Z_n(t) = n^{-1/2} \sum_{t=1}^{\lfloor nt\rfloor} X_t.
\end{equation}
Here and  for a real number $x$, $ \lfloor x\rfloor $ is the larger integer which does not exceed $x$. The  limiting behavior of $Z_n$ has attracted considerable interest in the  literature. In particular, it  has been shown,
 under a set of  weak dependent and moment conditions on $\{X_s, s\in\mathbb{Z}\}$,  
that as $n\rightarrow \infty$,  
\begin{equation} \label{eq.WeakConv}
 Z_n \Rightarrow W, 
 \end{equation}
where $ \Rightarrow$ denotes weak convergence and $ W$  is a Brownian motion in  ${\mathcal H}  $ characterized by the covariance operator $ C_W :{\mathcal H} \rightarrow {\mathcal H}$ of $W(1)$, which satisfies  
\begin{align} \label{eq.CovW}
\langle C_W(x),y \rangle & = \sum_{j=-\infty}^\infty \Cov(\langle X_0,x\rangle, \langle X_j, y \rangle) 
%& = 2\pi \sum_{r=1}^\infty \sum_{s=1}^\infty f_{r,s}(0) \langle v_r,x\rangle \langle v_s,y\rangle, \label{eq.CovOp}
\end{align} 
for any $ x,y\in{\mathcal H}$.  We refer here to   \cite{jirak2013weak}  and to  \cite{chen1998central}  for establishing such a result under different short-range dependence conditions.
%The complex-valued functions  $f_{r,s}$ appearing in  (\ref{eq.CovOp}) are the cross-spectral densities  of the two score processes $\{ \xi_{r,t}, t\in\mathbb{Z} \}$ and $\{ \xi_{s,t}, t\in \mathbb{Z}\}$,  where 
%$\xi_{j,t} =\langle X_t,v_j\rangle $ and $ v_j$, $j=1,2, \ldots$  denote the (up to a sign chosen) orthonormalized eigenfunctions associated to the eigenvalue $ \lambda_j$, $j=1,2, \ldots$,   of the lag zero autocovariance operator ${\mathcal C}_0=E(X_t-EX_0)\otimes (X_t-EX_0)$.

The above asymptotic result  is difficult to implement in practice due to the complicated structure of the covariance  of the limiting Brownian motion. One way to tackle this problem is to apply bootstrap or resampling techniques, which 
are  able  to  correctly imitate the random behavior of $ Z_n$.  For  functional time series, different bootstrap methods have been proposed that intent to properly mimic the temporal dependence structure of $ X_1,X_2, \ldots, X_n$; see  \cite{shang2018bootstrap} for an overview. \cite{PR94} applied the stationary bootstrap to Hilbert-space-valued processes,  \cite{DSW15} considered  the non-overlapping block bootstrap, \cite{franke2019residual} and \cite{zhu2017kernel} investigated properties of  residual-based bootstrap procedures for first order functional autoregression and \cite{PAP} developed a functional sieve bootstrap (FSB) approach. Especially, the  FSB builds upon the  Karhunen-Loeve representation of the random element $ X_s$ and  uses   a finite set of (static) functional principal components (scores), the temporal dependence  of which is mimicked  via  fitting  a  finite order vector autoregressive (VAR) model to the corresponding  (estimated) vector time series of scores.   The vector time series of scores  is then bootstrapped using  the  fitted  VAR model and  the procedure generates  fully  functional pseudo observations $X_1^\ast, X_2^\ast, \ldots, X_n^\ast$. 
%To achieve consistency, however, the number of principal components (scores)  used as well as the order of the VAR model fitted, both  have to  grow to infinity  with the sample size in order to properly capture the infinite dimensionality of the  process of scores as well as the infinite past of the  temporal dependence.  

The  first aim of this  paper  is to  justify   theoretically the use of the  FSB, when 
applied   to estimate the distribution of the partial sum process $ (Z_n(t))_{ t\in[0,1]}$.   Note  that,  despite the fact that  the technical arguments used   partly build  upon  some basic results proven  in 
\cite{PAP} and  \cite{paparoditis2023bootstrap}, new techniques are used here due  to the fact that  statements  over   uniform convergence  in the interval   $ [0,1]$   have to be established and  not  only in 
  the usual  $ \|\cdot\|$ norm of the associated  Hilbert space; see  Lemma 5.1 to 5.3 of  the Appendix.  Moreover,  the proof of the functional central limit theorem given in Lemma 5.4   of the Appendix  uses the basic Theorem 3.2 of   
   Billingsley (1968).   Note that the FSB (like its finite dimensional analogue, the AR-sieve bootstrap) can be valid  for approximating the distribution of a statistic of interest even if the underlying functional process  is not  linear   provided  the limiting  distribution of this statistic only depends on the  first and second order moments of the underlying process.   A theoretical justification of this  statement in   the finite dimensional case  in given in \cite{kreiss2011range}.   Although a general  proof  of this statement in the functional context is beyond  the scope of the current paper, as we will see, this  holds  true  for the  partial sum  statistic (\ref{eq.DefPS}). In particular,  Assumption 2.1  below, allows for a variety of   linear and nonlinear functional processes.

The established asymptotic validity of the FSB  applied to  $ (Z_n(t))_{ t\in[0,1]}$ enables the application of the same bootstrap procedure  to related  statistical inference problems. More specifically, we  
%There are two main approaches to do statistical inference for functional time series. One approach is to apply some form of dimension reduction and then use methods developed for multivariate time series analysis.  Alternatively, one might use a fully  functional approach. To be more specific and  in order to 
 consider the problem of testing  whether the   observed functional time series  $X_1,...,X_n$  has a change in mean. For this   testing problem,   a functional version of the CUSUM-test statistic can be used,  which is given by
% one can use 
\begin{equation} \label{eq.CUSUM}
T_n=\max_{k=1,...,n-1}\frac{1}{\sqrt{n}}\Big\|\sum_{i=1}^k X_i-\frac{k}{n}\sum_{i=1}^nX_i\Big\|.
\end{equation}
%where, recall,  $\|\cdot\|$ denotes the $L^2$ norm.  
Note that in the case of functional data, such a test statistic as well as variants thereof,  have been studied by \cite{horvath2014testing}, \cite{STW16}, \cite{AUE}.
Since the  CUSUM test statistic  (\ref{eq.CUSUM}) is a nonlinear functional of the partial sum process $(Z_n(t))_{t\in[0,1]}$, 
%where the latter is given by 
%\begin{equation*}
%Z_n(t)=n^{-1/2}\sum_{i=1}^{\lfloor nt\rfloor}Y_i, \ \ \mbox{for $t\in[0,1]$}.
%\end{equation*}
it  turns out  that    difficulties associated with implementing the asymptotic result (\ref{eq.WeakConv}) are transferred   to  difficulties in   obtaining  critical values of  the test $T_n$.
Among other approaches,  bootstrap or resampling methods   have  also been used: \cite{STW16} studied a sequential non-overlapping block bootstrap, \cite{dette2020functional} a  block multiplier bootstrap  and \cite{wegner2024robust} a  dependent wild bootstrap approach. 

These considerations justify the  second aim of this paper  which is  to   investigate the capability  of the  FSB   when applied to estimate the distribution of  $T_n$ under the null-hypothesis of no change and to deduce critical values of the test. 
%The aim of this paper is to justify theoretically the use of the  FSB  applied to estimate the random behavior of  the partial sum process $ Z_n(t), t\in[0,1]$. This  will allow for the use 
%of this bootstrap approach in order to  deduce  critical values of  the CUSUM based test statistic (\ref{eq.CUSUM}) in a fully functional context.

The paper is organized as follows:  Section 2   introduces  some notation and establishes the main theoretical result of this paper which shows validity  of the FSB in consistently estimating the distribution of the partial sum process. 
Section 3  discusses the  problem of  testing for a change point using the CUSUM-test $T_n$  and shows  consistency of the FSB under the null-hypothesis as well as  under local alternatives.  Section 4  investigates 
the finite sample performance of the FSB-based CUSUM-test  in the context  of a simulation study and comparisons to some alternative approaches  also are made. Technical proofs and additional simulation results are deferred to the supplementary file.
%Section 5.

\section{Notation and Main Result} \label{sec.Sec2}

Assume that for each $ s\in  \mathbb{Z}$,    the random element $X_s$ takes  values in the Hilbert space $ {\mathcal H}$ of square integrable functions from $[0,1]$ to $\mathbb{R}$ equipped with the inner product  $ \langle f, g \rangle=\int_0^1 f(u)g(u)du$, $f,g\in {\mathcal H}$  and the norm $\|f\|=\sqrt{\langle f,f\rangle}$. We denote by $ \E X_n \in{\mathcal H}$ the expectation and for $h\in \mathbb{Z}$ by $ {\mathcal C}_h =\E (X_n-\E X_0)\otimes (X_{n+h}-\E X_0)$ the lag $h$ autocovariance operator of the process ${\mathcal X}$, where  the tensor operator is defined as $x\otimes y = \langle x,\cdot\rangle y $ for $x,y \in {\mathcal H}$.  For a nuclear  (trace class) operator $ L$, $\|L\|_N$ denotes the nuclear norm and $ \|L\|_{HS}$ the Hilbert-Schmidt norm, if $ L$ is a Hilbert-Schmidt operator.  Assume that $ \sum_{h\in \mathbb{Z}} \Vert {\mathcal C}_h\Vert_N <\infty$, which implies that ${\mathcal X}$ possesses the  spectral density operator 
 \[ {\mathcal F}_\omega = \frac{1}{2\pi} \sum_{h\in \mathbb{Z} } \mathcal{C}_h e^{-i h \omega}, \ \ \ \omega \in (-\pi,\pi],\]
where ${\mathcal F}_\omega$  is  a continuous in $\omega$ and bounded; see  \cite{PAN} and \cite{hormann2015dynamic}.  Our aim  is to approximate the distribution of the partial sum process 
$ (Z_n(t))_{ t \in [0,1]}$ defined in (\ref{eq.DefPS}) using the   FSB  procedure. 

To elaborate,  
consider for 
  any $ m\in\mathbb{N}$,  the $m$-dimensional vector of scores 
$$ \xi_s(m)=(\xi_{s,1}, \xi_{s,2}, \ldots, \xi_{s,m})^\top, \ \ t\in \mathbb{Z}.$$
Here 
$\xi_{j,t} =\langle X_t,v_j\rangle $ and $ v_j$, $j=1,2, \ldots$  denote the (up to a sign chosen) orthonormalized eigenfunctions associated to the eigenvalue $ \lambda_j$, $j=1,2, \ldots$,   of the lag zero autocovariance operator ${\mathcal C}_0=E(X_t-EX_0)\otimes (X_t-EX_0)$. We assume that these  eigenvalues  are in descending order, that is  $ \lambda_1 >\lambda_2 > \lambda_3 > ...$, and that they are all distinct. 
 Observe that $\{ \xi_t(m), t\in \mathbb{Z} \}$ obeys  a so-called  vector autoregressive (VAR)  representation, see \cite{cheng1993baxter} and \cite{PAP}, that is, 
\begin{equation} \label{eq.xi-VAR}
 \xi_t(m) = \sum_{j=1}^\infty A_j(m) \xi_{t-j}(m) + e_t(m), \ \  t \in \mathbb{Z},
\end{equation}
where  $\{e_t(m)=(e_1(m), e_2(m), \ldots, e_{m}(m))^\top, t\in\mathbb{Z}\}$  is a $m$-dimensional, white noise process with mean zero and covariance matrix $ \Sigma_e(m)$. We write for short    $ e_t(m)\sim  WN(0,\Sigma_e(m))$. The sequence $ \{A_{j}(m),j\in \mathbb{N}\}$ of $m\times m$ coefficient matrices  satisfies $ \sum_{j\in\mathbb{N}} \|A_j(m)||_F <\infty$.
Truncating  the well-known Karhunen-Loeve  representation we can write  
 \begin{equation}
 \label{eq.FSB-dec}
  X_t=\sum_{j=1}^m \xi_{j,t}v_j + U_{t,m}, 
  \end{equation}
 where  $U_{t,m} =\sum_{j=m+1}^\infty \xi_{j,t}v_j.$  In the above decomposition, we  consider $ X_{t,m}:= \sum_{j=1}^m \xi_{j,t}v_j$ as the  main ``driving force'' of the random element $X_t$ and treat  the ``remainder'' $U_{t,m}$ as a noise term; see \cite{PAP}.
 The FSB uses expression (\ref{eq.FSB-dec}) to  generate new functional pseudo time series $X_1^\ast, X_2^\ast, \ldots, X_n^\ast$ which is  then applied to  estimate distribution of the  partial sum process (\ref{eq.DefPS}).
 The corresponding  algorithm  is described in  the next  section. 
 % for  change-point detection 
 %when the test statistic $ T_n$ is used. 

%\vspace*{0.25cm}
%\setdefaultleftmargin{0,75em}{2em}{}{}{}{}
%%\begin{bootstrap}  
%({\it  Bootstrap Algorithm}) \label{bo.model}
%\renewcommand{\baselinestretch}{0.8}
%\small\normalsize
\subsection{ The FSB  Proposal} \label{sec.FSB}
\begin{enumerate}
\setlength{\itemsep}{3pt}
\item[]\hspace*{-0,75em}{\bf Step 1:} \  Select a non-negative  integer $m$ and denote by 
$$  \widehat{\xi}_s(m)=(\widehat{\xi}_{s,1},\widehat{\xi}_{s,2}, \ldots,\widehat{\xi}_{s,m})^\top,  \ \  t=1,2, \ldots, n,$$ 
the vector of estimated  scores, $ \widehat{\xi}_{j,s}=\langle X_s,\widehat{v}_j \rangle$, where $ \widehat{v}_j$ denotes the estimated (up to a sign) orthonormalized eigenfunction associated to the estimated 
eigenvalue $\widehat{\lambda}_j$, $j=1,2, \ldots, m$, of the sample lag zero autocovariance operator $\widehat{\mathcal C}_0=n^{-1}\sum_{t=1}^n (X_t-\bar{X}_n) \otimes (X_t-\bar{X}_n)$. 
\item[]\hspace*{-0,75em}{\bf Step 2:} \  Select an order $p$ and  fit  to  the estimated series of scores $   \widehat{\xi}_t(m) $, $ t=1,2, \ldots, n$,    the VAR(p) model
\[   \widehat{\xi}_t(m) =\sum_{j=1}^p \widehat{A}_{j}(m)   \widehat{\xi}_{t-j}(m) + \widehat{e}_t(m),\]
$t=p+1, p+2, \ldots, n$, where $ \widehat{A}_{j}(m)$, $j=1,2, \ldots, p$, are the Yule-Walker estimators; see  \cite{brockwell1991time}, Chapter 11.
\item[]\hspace*{-0,75em}{\bf Step 3:} 
Generate pseudo random elements  $ X^\ast_1, X^\ast_2, \ldots, X^\ast_n$, as
$$ X_t^\ast= X_{t,m}^\ast + U_{t,m}^\ast,$$
where  the two functional   components  $ X_{t,m}^\ast$ and $ U^\ast_{t,m}$ appearing above are generated as follows:
\begin{enumerate}
\item[(i)]
  $X_{t,m}^\ast =\sum_{j=1}^m  \xi_{j,t}^\ast \hat{v}_j$  with  $ \xi_{j,t}^\ast$ the $j$th component of the vector  $  \xi^\ast_{t}(m)=(\xi_{1,t}^\ast, \xi_{2,t}^\ast, \ldots, \xi_{m,t}^\ast)^\top$,
\[ \xi_{t}^\ast (m)=\sum_{j=1}^p \widehat{A}_{j}(m) \xi_{t-j}^\ast (m)  +e_t^\ast(m),\]
with  $ e_t^\ast(m)$ pseudo  innovations  generated by i.i.d. resampling   from the  empirical distribution function of the centered residuals. That is, define $ \widetilde{e}_{t}(m) =\widehat{e}_{t}(m) - \overline{e}(m)$, where  
$ \overline{e}(m)=(n-p)^{-1}\sum_{t=p+1}^n \widehat{e}_t(m)$ and let $I_1,....,I_n$ be i.i.d. uniformly on $\{p+1,...,n\}$. Then set $e_t^\ast(m)=\widetilde{e}_{I_t}(m)$. 
 \item[(ii)] $ U_{t,m}^\ast$ is obtained by i.i.d. resampling from the set of  estimated  and centered "functional remainders"   $\widehat{U}_{t,m}-  \overline{U}_{m}$, $t=1,2, \ldots, n$, where $ \widehat{U}_{t,m}=X_t-\widehat{X}_{t,m} $,   
$ \widehat{X}_{t,m} = \sum_{j=1}^m \widehat{\xi}_{j,t} \widehat{v}_j$ and  $ \overline{U}_{m}= n^{-1}\sum_{t=1}^n \widehat{U}_{t,m}$. 
\end{enumerate}
\item[]\hspace*{-0,75em}{\bf Step 4:} \ Define  $ Z_{n,m}^\ast:= (Z^\ast_n(t))_{ t\in[0,1]}$, where
\begin{equation}  \label{eq.DefPS-Boot}
Z_{n,m}^\ast (t) = n^{-1/2} \sum_{t=1}^{\lfloor nt\rfloor} X^\ast_t.
\end{equation}
\end{enumerate}

%As it is common for  the bootstrap, the critical value $ C^\ast_{1-\alpha}$ given  in Step 5 can be estimated  by means of Monte Carlo.  

\subsection{Bootstrap validity}

To investigate   the  consistency  behavior of  $ Z^\ast_{n,m}$ as an estimator  of  the distribution of $ Z_n$,  some assumptions have to be made regarding the stochastic structure of  the underlying process 
$\{X_s, s\in\mathbb{Z}\}$ and the behavior  of the  FSB tuning parameters $m$ and $p$.  
Toward this goal,   we make use of   the   fourth order cumulant operator  of the functional process $ {\mathcal X}$  which is  defined for any $h_1,h_2,h_3\in {\mathbb Z}$,  as 
\begin{align*}
cum& ( X_{h_1}, X_{h_2},X_{h_3},X_0) =E[(X_{h_1}\otimes X_{h_2})\otimes(X_{h_3}\otimes X_0)] \\
& - E[X_{h_1}\otimes X_{h_2}]\otimes E[X_{h_3}\otimes X_0]-E[X_{h_1}\otimes X_{h_3}]\otimes_{op} E[X_{h_2}\otimes X_0] \\
& - E[X_{h_1}\otimes X_{0}]\otimes^\top_{op} E[X_{h_2}\otimes X_3]
\end{align*}
In the above expression and  for linear operators $L_j :  {\mathcal H}\rightarrow {\mathcal H}$,  $j=1,2,3$, the following definitions are used:     $ L_1\otimes_{op}L_2(L_3) :=L_1L_3 L_2^\ast $ and $ L_1\otimes^\top_{op}L_2(L_3) := L_1L_3^\top L_2^\top$, where  $ L^\ast$ is  the adjoint operator and $ L^\top$ the transposed operator of   $ L$; see \cite{rademacher2024asymptotic} for  more details.

We start with the following assumption which  summarizes our requirements on the properties   of the underlying  functional process and 
which uses the notion of $L^p$--$M$ approximability; see \cite{hoerkok09} and \cite{berkes2013weak}. To elaborate, suppose that for the functional process $ \{X_s,s\in {\mathbb Z} \}$   
the random element $X_s$  admits the representation $X_s = f(\varepsilon_s, \varepsilon_{s-1},\ldots )$,  where  the $\varepsilon_s$'s   are i.i.d. random elements in ${\mathcal H}$,  $f$ is some measurable function
$ f : {\mathcal H}^\infty \rightarrow {\mathcal H}$ and $ E\|X_s\|^p<\infty$ for some $p\in {\mathbb N}$.  If for an independent copy $\{  \widetilde{\varepsilon}_s, s\in {\mathbb Z}\} $   of  $ \{ \varepsilon_s, s\in {\mathbb Z}\}$ and 
\begin{equation*}
X_s^{(M)} = f(\varepsilon_s, \varepsilon_{s-1},\ldots,  \varepsilon_{s-M-1}, \widetilde{\varepsilon}_{s-M}, \widetilde{\varepsilon}_{s-M-1},\ldots ),\end{equation*}
%X_t^{(M)} = f(\varepsilon_t, \varepsilon_{t-1}, \ldots,  \varepsilon_{t-M+1},\widetilde{\varepsilon}_{t-M},\widetilde{\varepsilon}_{t-M-1},\ldots )$,
 the condition $ \sum_{k=1}^\infty \big(E\|X_k-X_k^{(M)} \|^p\big)^{1/p} <\infty$ is satisfied, then the process  ${\mathcal X}$  is called $L^p$--$M$ approximable. 
%Lp?M approximability is a notion of weak dependence, which applies to many commonly used functional time series models, like linear functional processes, functional ARCH processes, etc.; see H?rmann and Kokoszka (2010) for more details
%\emph{insert definition}

\begin{assumption}\label{ass1}
\begin{enumerate}
\item[(i)] \ ${\mathcal X}=(X_s,s\in\mathbb{Z})$ is purely nondeterministic, $ L^4$-$M$ approximable process satisfying
\[ \sum_{h\in \mathbb{Z}} (1+|h|)\|C_h\|_{N} <\infty \ \text{ and } \ \sum_{h_1,h_2,h_3\in  \mathbb{Z}}\| Cum_{h_1,h_2,h_3}\|_N <\infty,\]
where  $ Cum_{h_1,h_2,h_3} =cum(X_{h_1}, X_{h_2}, X_{h_3},X_0)$ is the fourth order cumulant operator of $ {\mathcal X}$.
\item[(ii)] The spectral density operator $ {\mathcal F}_\omega $ of the process $ {\mathcal X}$ is of full rank, that is, $ {\mbox ker}({\mathcal F}_\omega)=0$, for all $\omega \in [0,\pi]$. 
\item[(iii)] For any $m\in {\mathbb N} $,  let $ G_{e}^{(m)}$  be  the marginal distribution function of $e_t(m)$ (which by  part (i) of the assumption  does not depend on $t$).  
For any   $K\in {\mathbb N}$, $K<m$,  denote by  $ G_{e,K}^{(m)}$ the distribution function of the first $K$  components of  the vector $e_t(m)$, that is of the vector $ (e_1(m), e_2(m), \ldots, e_{K}(m))^\top$. 
Then, as $ m \rightarrow \infty$,  $ G_{e,K}^{(m)} \rightarrow G_{e,K} $,  where   $G_{e,K}$ is continuous. 

%\item[(iv)]  It holds true that 
%\begin{equation} \label{eq.Lem3-cond}
% \sum_{l=1}^{\infty} \sum_{h=-\infty}^\infty (1+|h))  |\gamma_\ell(h)| <\infty,
% \end{equation}
% and that $m\rightarrow\infty$ such that  
%  \begin{align} \label{eq.Lem3-1}
%  (\log_2(2n))^2 \sum_{l=m+1}^\infty \sum_{h=-\infty}^{\infty}|\gamma_\ell(h) | \rightarrow 0, 
% \end{align}
\end{enumerate}
\end{assumption}

%\vspace*{0.3cm}

\begin{rem}
\begin{enumerate}
\item[]
\item[(i)]\ 
Observe that  $ \gamma_{l_1,l_2}(h) :=  Cov(\xi_{l_1,0},\xi_{l_2,h}) = \langle C_h(v_{l_1}),v_{l_2}\rangle $ and therefore, Assumption \ref{ass1}(i) implies that for all $ l_1,l_2 \in {\mathbb N}$
\[ \sum_{h\in{\mathbb Z}}(1+|h|) |\gamma_{l_1,l_2}(h)| =  \sum_{h\in{\mathbb Z}}(1+|h|)  |\langle C_h(v_{l_1}),v_{l_2}\rangle| \leq \sum_{h\in{\mathbb Z}}(1+|h|)  \Vert C_h \Vert_N <\infty. \]
\item[(ii)] \ 
 Let  $ cum(\xi_{l_1,h_1}, \xi_{l_2,h_2},\xi_{l_3,h_3},\xi_{l_4,0})$ be the fourth order cumulant of the  scores processes $ \{\xi_{l_j,t}, t \in\mathbb{Z}\}$, $ l_1,l_2,l_3,l_4 \in \mathbb{N}$.  Then, Assumption \ref{ass1}(ii) implies that 
\[ 
\sum_{h_1,h_2,h_3 \in {\mathbb Z}}|cum(\xi_{l_1,h_1}, \xi_{l_2,h_2},\xi_{l_3,h_3},\xi_{l_4,0})|  < \infty.\]
This follows because, 
\begin{align*}
\big|cum(\xi_{l_1,t_1}, & \xi_{l_2,t_2},\xi_{l_3,t_3},\xi_{l_4,t_4})\big| \\ 
&=\big|cum(\langle X_{t_1}, v_{l_1}\rangle , \langle X_{t_2},v_{l_2}\rangle, \langle X_{t_3},v_{l_3}\rangle, \langle X_{t_4},v_{l_4}\rangle)\big|\\
& =| \langle cum(X_{t_1},X_{t_2},X_{t_3},X_{t_4}), (v_{l_1}\otimes v_{l_2}) \otimes (v_{l_3}\otimes v_{l_4}) \rangle |\\
& \leq \big\| cum(X_{t_1},X_{t_2},X_{t_3},X_{t_4})\big\|_{N} \big\| (v_{l_1}\otimes v_{l_2}) \otimes (v_{l_3}\otimes v_{l_4}) \big\|\\
&=\big\| cum(X_{t_1},X_{t_2},X_{t_3},X_{t_4})\big\|_{N}.
\end{align*}
\item[(iii)] \ Notice that when the dimension $m$  of  the vector autoregressive representation (\ref{eq.xi-VAR}) increases through adding new elements (scores) to the vector $ \xi_t(m)$,  the corresponding  vector of white noise innovations  $e_t(m)$ may entirely  change. Part (iii) of  Assumption \ref{ass1}
ensures  that  despite such changes, the  distribution function of any  fixed number of  the first  $K$ components of the vector of white noise innovations converges  to a continuous distribution function as $m$ increases to infinity.  
\end{enumerate}
\end{rem}

%\vspace*{0.2cm}

Additional to Assumption \ref{ass1},  the parameters $m$ and $p$ involved in the FSB algorithm have  to increase to infinity at a controlled rate, as the sample size $n$ increases to infinity. 
Recall that $m$ determines the  number of   principal components used to approximate the infinite dimensional  score process $ \xi_s=(\xi_{1,s}, \xi_{2,s}, \ldots)^\top$, while  $p$ determines the finite order of the VAR process used to approximate the infinite order VAR representation (\ref{eq.xi-VAR}).  In order to capture this  infinity dimensional nature      of both components,   the parameters $ m$ and $ p$ have to increase to infinity with  the sample size $n$. This however has to  be done in a proper way which has to  take into account a number of  issues including  the dependence characteristics of the underlying process ${\mathcal X}$, the fact that the parameter matrices $A_{j,p}(m) $ of the VAR(p) process have to be  estimated and   that  the corresponding  estimates are based on the estimated  scores $ \widehat{\xi}_{j,s}$  and not on the unobserved random variables  $\xi_{j,s}$, $j=1,2 \ldots, m$,  $s=1,2, \ldots, n$. 
Our requirements concerning this part of the FSB algorithm  are summarized in the following  assumption.

%\vspace*{0.2cm}

\begin{assumption}\label{ass2} The sequences $ m=m(n)$ and $ p=p(n)$ increase to infinity as $n$ increases to infinity such that:
\begin{enumerate}
\item[(i)] \ $ \frac{\displaystyle m^{3/2}}{\displaystyle p^{1/2}} =O(1)$.
\item[(ii)] \  $\frac{\displaystyle p^7}{\displaystyle \sqrt{n}\lambda_m^2} \sqrt{\sum_{j=1}^m \alpha_j^{-2} } \rightarrow 0$, where $ \alpha_1=\lambda_1-\lambda_2$ and $ \alpha_j =\min\{\lambda_{j-1}-\lambda_j, \lambda_j-\lambda_{j+1}\}$ for $ j=2,3, \ldots, m$.
\item[(iii)] \  $\delta_m^{-1}\sum_{j=p+1}^\infty j \|A_j(m)\|_F \rightarrow 0$, where $ \delta_m>0$ is the lower bound of the spectral density matrix $f_\xi$ of the $m$-dimensional  score process $ \{\xi_t, t\in {\mathbb Z}\}$.
\item[(iv)] \  Let $ \tilde{A}_{p,m}=(\tilde{A}_{j,p}(m), j=1,2, \ldots,p)$  be the estimators of $  (A_{j,p}(m),  j=1,2 \ldots, p)$, obtained by the same method as $ \hat{A}_{j,p}(m)$ but based on the time series of true scores $ \xi_1, \xi_2, \ldots, \xi_n$. Then,
$m^4p^2 \|\tilde{A}_{p,m}- A_{p,m}\|_F =\mathcal{O}_{P}(1)$. 
\end{enumerate}
\end{assumption}

%Remark XXXX ?
%To proceed with investigating the  validity of the FSB procedure, we 
Consider now  the bootstrap partial sum process $(Z^\ast_{n,m}(t))_{t\in[0,1]}$, where 
$Z^\ast_{n,m}(t)$ is defined in (\ref{eq.DefPS-Boot}) 
and  for $n\in {\mathbb N}$, the pseudo time series $ X_1^\ast, X_2^\ast, \ldots, X_n^\ast$, is generated as in Step 3 of  the bootstrap algorithm presented in Section~\ref{sec.FSB}. 
Then,   the following weak convergence result holds true:

%\vspace*{0.2cm}
\begin{thm} \label{th.boot} 
Under Assumptions \ref{ass1} and \ref{ass2} we have that, as $n\rightarrow\infty$, 
\[ Z^\ast_{n,m} \Rightarrow W, \text{ in probability},  \]
where $ W$ is a Brownian motion  in ${\mathcal H}$ and the  covariance operator of  $ W(1)$ coincides with the covariance operator  $ C_W$    given in  (\ref{eq.CovW}).
%
%Here,  $W$ is a Brownian Motion in $H$ with covariance operator $C_\omega$ satisfying
%\begin{align*}
%\langle C_\omega x,y \rangle = \sum_{j=-\infty}^\infty \Cov(\langle X_0,x\rangle, \langle X_j, y \rangle) \\
%= 2\pi \sum_{r=1}^\infty \sum_{s=1}^\infty f_{r,s}(0) \langle v_r,x\rangle \langle v_s,y\rangle,
%\end{align*} 
%where $f_{r,s}$ is the cross-spectral density of the two score processes $\{ \xi_{r,t}, t\in\mathbb{Z} \}$ and $\{ \xi_{s,t}, t\in \mathbb{Z}\}$ and defined as
%\[f_{r,s}(\omega) = \frac{1}{2\pi} \sum_{h=-\infty}^\infty \Cov(\xi_{0,r},\xi_{h,s})\exp(-ih\omega)  \]
\end{thm}
By the term  ``$\Rightarrow$ in probability", we mean that for any subseries $(n_k)_{k\in\mathbb{N}}$ of the natural numbers, there exists a subsubseries $(n_{k_l})_{l\in\mathbb{N}}$, such that the distribution of the bootstrapped partial sum process conditional on $X_1,X_2, \ldots, X_{n_{k_l}}$ converges weakly to the distribution of $W$ with probability 1 on this subsubseries.

%we mean   that there  exists a subset of  the  probability space where the functional time series $X_1,X_2, \ldots, X_n$ takes its values, such that  this set has high probability to be realized and for  every functional time series stemming from  this subset, the weak convergence  stated in the above theorem holds true.
%' we mean that for any metrization of weak convergence, e.g. the Prokhorov metric, the distance of the conditional distribution of the bootstrap process and the distribution of the Brownian motion converges to 0 in probability.

\section{Change Point Detection}
\label{sec.ChPoint}

\subsection{Behavior under the Null-Hypothesis} \label{sec.H0}

In this  section we  denote by  $Y_1,...,Y_n$ the functional observations at hand and  where we assume that they  are obtained as 
\begin{equation*}
Y_s=\begin{cases}
X_s \ \ \ &\text{for }s \leq k^\star\\
X_s+\mu&\text{for }s >k^\star
\end{cases}
\end{equation*}
for some (unknown) $k^\star\in\{1,...,n-1\}$ and for some $\mu\in {\mathcal H}$. Our interest is focused on the problem of testing the null-hypotheses ${\rm  H}_0$ of mean stationarity against the alternative ${\rm  H}_1$ of a single  change-point, that is
\begin{equation*}
{\rm  H}_0: \ \mu=0 \ \ \text{ against } \ {\rm  H}_1: \ \mu\neq 0.
\end{equation*}

Suppose that $H_0$ is true. Under this assumption we have $Y_s=X_s$ for $s=1,2, \ldots,n$ and the  functional CUSUM-test statistic can be rewritten as
\begin{align} \label{eq.teststat} 
T_n & =\max_{1\leq k <n} \frac{1}{\sqrt{n}} \Big\| \sum_{i=1}^kY_i -\frac{k}{n}\sum_{j=1}^n Y_j\Big\|  =\max_{1\leq k <n} \frac{1}{\sqrt{n}} \Big\| \sum_{i=1}^kX_i -\frac{k}{n}\sum_{j=1}^n X_j\Big\|  \\
& =\max_{1\leq k <n}\Big\|Z_n(k/n)-\frac{k}{n}Z_n(1)\Big\|.
\end{align} 
%with $Z_n(t)=n^{-1/2}\sum_{i=1}^{\lfloor nt \rfloor}X_i$.  It is known that, under different  regularity conditions on the underlying functional process $ \{X_t,t\in\mathbb{Z}\}$, the partial sum processes  $(Z_n(t))_{t\in[0,1]} $ converges to a $\mathcal{H}$-valued Brownian motion $W$. More precisely, 
Under Assumption~\ref{ass1}  of Section~\ref{sec.Sec2}, we have that  the weak convergence result (\ref{eq.WeakConv}) holds true.
%\begin{equation} \label{eq.WConX}
%\big(Z_n(t)\big)_{t\in[0,1]}\Rightarrow \big(W(t)\big)_{t\in[0,1]},
%\end{equation} 
%as $n\rightarrow \infty$; see Theorem  1.2 of \cite{jirak2013weak}. Other authors have proved such a  result  under other short-range dependence conditions, see e.g. \cite{chen1998central}. 
By the continuous mapping theorem, it  then follows   that if $H_0$ is true,  then 
\begin{equation} \label{eq.limit-Tn}
T_n \Rightarrow \sup_{t\in [0,1]} \| W(t) - t W(1)\|.
\end{equation}
%\begin{equation} \label{eq.limit-Z}
%Z_n \Rightarrow W,
%\end{equation}
%Here the Brownian motion  $\{W(t), t \in[0,1] \}$ in ${\mathcal H}$ is characterized by the covariance operator $C_W:{\mathcal H} \rightarrow {\mathcal H}$ of $W(1)$, which satisfies
Observe  that the covariance operator  $C_W:{\mathcal H} \rightarrow {\mathcal H}$ of $W(1)$, see   (\ref{eq.CovW}),  also can be written as 
\begin{equation*}
\langle C_W(x),y \rangle  = 2\pi \sum_{r=1}^\infty \sum_{s=1}^\infty f_{r,s}(0) \langle v_r,x\rangle \langle v_s,y\rangle, \label{eq.CovOp}
\end{equation*} 
for any $ x,y\in{\mathcal H}$. The complex-valued functions  $f_{r,s}$ appearing in the equation above are the cross-spectral densities  of the two score processes $\{ \xi_{r,t}, t\in\mathbb{Z} \}$ and $\{ \xi_{s,t}, t\in \mathbb{Z}\}$.  Notice that $f_{r,s}(\omega)  = \frac{1}{2\pi} \sum_{h=-\infty}^\infty \Cov(\xi_{0,r},\xi_{h,s})\exp(-ih\omega)   = \langle {\mathcal F}_\omega(v_r), v_s \rangle$
for $r,s\in{\mathbb N}$.

By the weak convergence result (\ref{eq.limit-Tn}),  an asymptotically  valid  testing procedure   at  any   level $ \alpha\in (0,1)$   is obtained by rejecting $H_0$ if  $ T_n \geq C_{1-\alpha}$, where   $ C_{1-\alpha}$  denotes  the  upper $\alpha$-quantile of the distribution of  $ \sup_{t\in [0,1]} \| W(t) - t W(1)\|$.\\  
%However,  
%Notice that $ C_{1-\alpha}$   refers to a distribution where 
% is difficult to obtain  in practice 
%expression  (\ref{eq.CovOp})  makes it  clear that  practical  implementation of the test  is difficult 
%due to the fact that  
%calculation  of the critical values $ C_{1-\alpha} $ is   difficult  due to the fact that 
%the 
%the limiting  Brownian motion  involved has a covariance structure which depends on the unknown spectral density operator $ {\mathcal F}_\omega$,
%that is,
% of the underlying process, i.e., 
 %on the infinite dimensional matrix of cross spectral densities $ \big(f_{r,s}(\omega)\big)_{r,s \in {\mathbb N}}$.  
 %Moreover, in finite samples, 
 %the limiting Gaussian approximation (\ref{eq.limit-Tn}) may  not properly capture  all aspects of the distribution of  $T_n$. 
 %These facts motivate the use of alternative, more specifically,  bootstrap  methods to deduce critical values of the test. 
Based on Theorem~\ref{th.boot}, we can now apply the FSB to estimate critical values of the $T_n$ test. In particular,  denote by $T_{n,m}^\ast$  the bootstrap analogue of $T_n$, which is given by  
\[ T_{n,m}^\ast = \max_{1\leq k <n} \frac{1}{\sqrt{n}} \big\| \sum_{s=1}^kX^\ast_s -\frac{k}{n}\sum_{s=1}^n X^\ast_s\big\| \]
and $ X_1^\ast, X_2^\ast, \ldots, X_n^\ast$ is generated as in the FSB algorith presented in Section~\ref{sec.FSB}. 
%\item[]\hspace*{-0,75em}{\bf Step 5:} \
Let   $ C^\ast_{1-\alpha}$  be the upper $\alpha$-percentage point of the distribution of $ T^\ast_{n,m}$, that is,    $ P^\ast(T^\ast_{n,m} \geq C^\ast_{1-\alpha}) =\alpha$. The  FSB-bsed test rejects then  $ H_0$, if  $ T_n \geq C^\ast_{1-\alpha}$. Theorem~\ref{th.boot} together with  the continuous mapping  lead to  the following result:
 
 \begin{cor}  \label{cor.FSB} Under the assumptions of Theorem~\ref{th.boot} it holds true as $n\rightarrow\infty$,  
\[  d_\infty\big(T^\ast_{n,m} , T_n\big) \rightarrow 0, \ \ \mbox{in probability},\]
   where $d_\infty$ denotes Kolmogorov's distance between the distributions of the  random variables $T_{n,m}^\ast$ and $ T_n$, respectively. 
 \end{cor}  
   
  Observe that  the above  validity of the FSB procedure for  consistently estimating the distribution  of   $T_n$  holds true whenever   the  statistic $T_n$  obeys the limiting behavior described in equation   (\ref{eq.limit-Tn}). Together with the continuity of the limit distribution of the random variable
  $ T_n$, this result  also implies that  
   \[ P(T_n \geq C^\ast_{1-\alpha}) \rightarrow \alpha,\]
 in probability, as $ n \rightarrow \infty$, that is,  the FSB  based test achieves (asymptotically) the desired level $\alpha$.

\subsection{Consistency under local alternatives} \label{sec.H1}

Let us now discuss   the local power properties of the FSB  based testing procedure presented in Section~\ref{sec.H0}.  For this consider  unobserved functional time series 
$X_1, \ldots ,X_n$  and denote  the observed  time series by  $Y_1, Y_2, \ldots Y_n$, where  $Y_i$ has  a change of order $ O(1/n^r)$ at some unknown time point $k^\ast$, that is,
\begin{equation} \label{eq.LocAlt}
Y_s=\begin{cases}X_s \ &\text{for }s\leq k^\ast\\X_s+ n^{-r} \mu \ &\text{for }s> k^\ast. \end{cases}
\end{equation}
Here $\mu\in {\mathcal  H}$ with $\mu\neq 0$, $r\in(0,1)$ and $k^\ast=\lfloor nt^\ast\rfloor$ for some fixed $t^\ast\in(0,1)$.  Notice that  in this set up, the testing problem becomes more difficult as $ n$ increases to infinity,  since the magnitude of the change  shrinks to zero.  
As before, consider  then the partial sum process $(Z_{n,X}(t))_{t\in[0,1]}$, with $Z_{n,X}(t)=n^{-1/2}\sum_{i=1}^{\lfloor nt \rfloor}(X_i-\E X_i)$. If $(Z_{n,X}(t))_{t\in[0,1]}\Rightarrow W$ as $n\rightarrow\infty$ and \eqref{eq.LocAlt}  with $r=1/2$ is satisfied, then  the following holds true for the  observed functional time series $Y_1,Y_2, \ldots, Y_n$:
\begin{equation} \label{eq.WConY}
\max_{k=1,...,n} \frac{1}{\sqrt{n}}\Big\|\sum_{i=1}^k(Y_i-\bar{Y}_n)\Big\|\Rightarrow \sup_{t\in[0,1]}\big\|W(t)-tW(1)+g(t)\mu\big\|,
\end{equation}
where  $\overline{Y}_n=n^{-1}\sum_{i=1}^n Y_i$ and 
\begin{equation*}
g(t)=\begin{cases}t(1-t^\ast) \ \ \text{for }t\leq t^\ast\\ t^\ast(1-t) \ \text{for }t> t^\ast. \end{cases}
\end{equation*}
The above result follows by the continuous mapping theorem and  Corollary 2  of \cite{STW16}. However, if \eqref{eq.LocAlt} holds with $r<1/2$, then $T_n=\max_{k=1,...,n} \frac{1}{\sqrt{n}}\|\sum_{i=1}^k(Y_i-\bar{Y}_n)\|$ converges to the same limit as under the ${\rm  H}_0$. Furthermore,  $T_n\rightarrow \infty$ for $ r>1/2$.

Now, let $Z_{n,Y}^\ast$ be the bootstrap version of the partial sum defined by  
\[  Z_{n,Y}^\ast(t)=\frac{1}{\sqrt{n}}\sum_{i=1}^{\lfloor nt \rfloor}Y_i^\ast,\] 
 where the bootstrap pseudo observations $Y_1^\ast, Y_2^\ast, \ldots, Y_n^\ast$  are generated  by applying to the 
 observed functional time series $Y_1, Y_2, \ldots, Y_n$   the same FSB  procedure as the one used  to generate  $X_1^\ast, X_2^\ast, \ldots, X_n^\ast$ in Section~\ref{sec.H0}. We then have the following result:

\begin{thm}\label{thm:altloc} Under the assumptions of Theorem \ref{th.boot} and  the validity of  Model \eqref{eq.LocAlt} with $r>1/4$, we have that
\begin{equation*}
\left(Z_{n,Y}^\ast(t)\right)_{t\in[0,1]}\Rightarrow W
\end{equation*}
and 
\begin{equation} \label{eq.WConYBoot}
T_{n,m}^\ast:= \max_{k=1,...,n} \frac{1}{\sqrt{n}}\Big\|\sum_{i=1}^k(Y_i^\ast-\bar{Y}_n^\ast)\Big\|\Rightarrow \sup_{t\in[0,1]} \big\|W(t)-tW(1) \big\|,
\end{equation}
in probability. 
\end{thm}

By comparing (\ref{eq.WConYBoot}) with (\ref{eq.WConY}) and (\ref{eq.limit-Tn}) one sees that, even under the  sequence of local alternatives  (\ref{eq.LocAlt}), the  FSB procedure  manages to consistently estimate the distribution that the test statistic $T_n$ would have under  ${\rm  H}_0$.  This immediately  implies    consistency  of  the  change-point test $T_n$ based on the bootstrap critical values  $ C^\ast_{1-\alpha}$:

\begin{cor} Under the Assumption of Theorem \ref{th.boot}  and for a functional time series $(Y_n)_{n\in\mathbb{N}}$ satisfying \eqref{eq.LocAlt} with $r\in(\frac{1}{4},\frac{1}{2})$ , it holds true that, in probability, 
\[ \lim_{n\rightarrow\infty}P(T_n\geq C^\ast_{1-\alpha}) =1. \]
\end{cor}

\section{Numerical Results}\label{sec:num}

In this section  we compare the size and the power of  the CUSUM-test with critical values obtained using  different methods: The  FSB procedure  introduced in this paper\footnote{The R-Code of the FSB  is available under \texttt{https://cloud.ovgu.de/s/BHPi8b3e99RDcY4}.}, the non-overlapping block bootstrap considered in  \cite{STW16} and a testing procedure based on  estimation of parameters involved in  the limit distribution which has been  proposed by \cite{AUE}.

Let us first note that while the computational complexity of the CUSUM statistic grows linearly in the sample size $n$, the computation time for the bootstrap methods nevertheless can be quite long, because these methods rely on Monte Carlo evaluation.  For one sample of size of $n=200$, $50$ grid points to calculate the integrals and $1000$ bootstrap iterations, we measured a computation time of 31.7s for the functional sieve bootstrap and 33.2s for the non-overlapping block bootstrap (on a standard laptop). The method by \cite{AUE} avoids Monte Carlo evaluation and is thus much faster: we measured 3.3s.

For all three methods, tuning parameters have to be chosen. For the  FSB, we choose the number of principal components $m$ and the order $p$ of the bootstrap VAR-model as outlined in \cite{paparoditis2023bootstrap}: the minimum number of principal is chosen to explain at leat 85\% of the variance, and for the selection of the order $p$,  a corrected Akaike information  criterion is used. Some additional simulations to illustrate the effect of the autoregressive order on the FSB can be found in the Appendix. For the non-overlapping block bootstrap (NBB), we choose the block length by adapting a method by \cite{RS17}, see also \cite{wegner2021block}. This method is also applied for the procedure as proposed by \cite{AUE}, who also kindly provided their R-codes.

Different stochastic models are used  to create the functional observations. In particular,  a functional first order autoregressive (FAR(1))  model with Brownian bridge innovations as in \cite{STW16}, a FAR(1) process with squared Brownian bridges as innovations, and  a functional moving average (FMA(1)) process of order 1 as in \cite{aue2017estimating} have been used. In all scenarios, the results are based on the sample sizes  $n=50,100,200$ and  the rejection frequencies are based on 2000 simulation runs and 1000 bootstrap samples.

First, we generate Gaussian FAR(1)  processes by
\begin{equation*}
X_{n+1}(t)=C\int_{0}^{1}stX_n(s)ds+\epsilon_{n+1}(t),
\end{equation*}
where $(\epsilon_n)_{n\in \mathbb{N}}$ is a independent identically distributed (i.i.d.) series of Brownian bridges. The strength of the dependence is regulated by the parameter $C$ which we choose either as 0.245, 0.49 or as -0.49. For this model, the FSB procedure  keeps the size best, see Tables \ref{tabSize50} to \ref{tabSize200}, while the NBB and the asymptotic method lead to oversized tests for positive $C$ and conservative tests for negative $C$.

Next, we study the behaviour for non-Gaussian FAR(1)-processes generated by 
\begin{equation*}
X_{n+1}(t)=C\int_{0}^{1}stX_n(s)ds+\epsilon_{n+1}^2(t)+\eta^2(t),
\end{equation*}
where $(\epsilon_n)_{n\in\mathbb{N}}$ and $(\eta_n)_{n\in\mathbb{N}}$ are independent i.i.d. sequences of Brownian bridges and $C=0.490$. For this time series, again the  FSB method holds the size, while the rejection frequency especially for the asymptotic method is to high for $n=100$ or $n=200$. 
 
Finally, we simulate a  FMA(1) process of order 1 which is constructed like in \cite{aue2017estimating}.  First, for every simulation run a 21$\times$21 matrix $A$ with independent, centered Gaussian entries and $\Var[A_{ij}]=(ij)^{-\gamma}$ is generated and standardized to have spectral norm 1. We choose either $\gamma=1$ (fast decay of autocovariances) or $\gamma=0.6$ (slow deay of autocovariances). Next,  the vector-valued process 
\begin{equation*}
Z_n=\epsilon_n+A\epsilon_{n-1}
\end{equation*}
is simulated for an i.i.d. sequence $(\epsilon_n)_{n\in\mathbb{N}}$ of centered Gaussian random vectors with $\Var(\epsilon_{n,i})=i^{-1}$. We then created a  FMA(1) process $(X_n)_{n\in\mathbb{N}}$ by using the entries of $Z_n$ as Fourier coefficients of $X_n$. In this case, the NBB exceeds the theoretical size most often, while the FSB and the asymptotic method keep the size most of the time. The FSB is oversized for the fast decay of autocovariances ($\gamma=1$) and $n=100$, while the asymptotic method is oversized for the slow decay of autocovariances ($\gamma=0.6$) and $n=200$.

In summary, we see that the FSB is less often oversized under the null-hypothesis in the scenarios we have investigated, especially compared to the non-overlapping bootstrap.

\begin{table}
\caption{Empirical rejection frequencies under $H_0$  for theoretical sizes $\alpha$, sample size $n=50$ (FSB = functional sieve bootstrap, NBB = non-overlapping block bootstrap, Asymptotic = method by \cite{AUE}).}\label{tabSize50}
\vspace*{0.2cm}
\begin{tabular}{|l|l|ccc|}
\hline 
\rule[-1ex]{0pt}{2.5ex} model & method  & $\alpha=10\%$ & $\alpha=5\%$ & $\alpha=1\%$ \\ 
\hline 
FAR(1)  & FSB &  0.092 & 0.035 &  0.001 \\ 
Gaussian & NBB &0.106  & 0.049 &  0.008 \\ 
$C=0.245$ &  Asymptotic & 0.106 & 0.039 & 0.006 \\ 
\hline 
FAR(1)  &  FSB & 0.085 & 0.028& 0.002\\ 
Gaussian & NBB & 0.127 & 0.055 & 0.010\\ 
$C=0.49$ & Asymptotic & 0.114 & 0.051 & 0.004 \\ 
\hline
FAR(1) &  FSB & 0.102 & 0.036 & 0.004\\ 
Gaussian  & NBB & 0.075 & 0.030 & 0.004\\ 
$C=-0.49$  & Asymptotic & 0.067 & 0.019 & 0.001 \\ 
\hline
FAR(1) &  FSB & 0.077 & 0.025 & 0.002\\ 
non-Gaussian & NBB & 0.091 & 0.046 & 0.003\\ 
$C=0.49$ & Asymptotic & 0.094 & 0.037 & 0.003 \\ 
\hline
FMA(1) &  FSB & 0.098 & 0.028 & 0.000 \\ 
fast decay &  NBB & 0.108 & 0.052 & 0.005 \\ 
of autocov. & Asymptotic & 0.098 & 0.043 & 0.002 \\  
\hline
FMA(1) & FSB & 0.082 & 0.030 & 0.001\\ 
slow decay &  NBB & 0.101 & 0.049 & 0.012\\ 
of autocov. & Asymptotic &  0.088 &0.040  & 0.006 \\ 
\hline 
\end{tabular} 
\end{table}

\begin{table}
\caption{Empirical rejection frequencies under $H_0$ for theoretical sizes $\alpha$, sample size $n=100$ (FSB = functional sieve bootstrap, NBB = non-overlapping block bootstrap, Asymptotic = method by \cite{AUE}).}\label{tabSize100}
\vspace*{0.2cm}
\begin{tabular}{|l|l|ccc|}
\hline 
\rule[-1ex]{0pt}{2.5ex} model & method  & $\alpha=10\%$ & $\alpha=5\%$ & $\alpha=1\%$ \\ 
\hline 
FAR(1)  & FSB & 0.088 & 0.041 & 0.005 \\ 
Gaussian  & NBB & 0.113 & 0.058 & 0.008 \\ 
$C=0.245$ &  Asymptotic & 0.112 & 0.045 & 0.005 \\ 
\hline 
FAR(1)  &   FSB & 0.092 & 0.039 & 0.006 \\ 
Gaussian  & NBB & 0.132 & 0.062 & 0.015 \\ 
$C=0.49$ & Asymptotic & 0.131 & 0.055 &  0.009 \\ 
\hline
FAR(1)  & FSB & 0.078 & 0.039 & 0.005 \\ 
Gaussian  &  NBB & 0.064 & 0.028 & 0.004 \\ 
$C=-0.49$  & Asymptotic & 0.058 & 0.023 &  0.002\\ 
\hline
FAR(1)  &  FSB & 0.089 & 0.039  &0.003 \\ 
non-Gaussian  & NBB & 0.114 & 0.055 & 0.008 \\ 
$C=0.49$  & Asymptotic & 0.130 & 0.064 & 0.006 \\ 
\hline
FMA(1) &  FSB &0.117 & 0.055 &  0.004\\ 
fast decay &  NBB & 0.111 & 0.048 & 0.010\\ 
of autocov. & Asymptotic & 0.082 & 0.042  & 0.010 \\  
\hline
FMA(1) & FSB & 0.090 & 0.035 & 0.001\\ 
slow decay &  NBB & 0.110 & 0.055 & 0.008 \\ 
of autocov. & Asymptotic &  0.099 & 0.050 & 0.007 \\ 
\hline 
\end{tabular} 
\end{table}

\begin{table}
\caption{Empirical rejection frequencies under $H_0$ for theoretical sizes $\alpha$, sample size $n=200$ (FSB = functional sieve bootstrap, NBB = non-overlapping block bootstrap, Asymptotic = method by \cite{AUE}).}\label{tabSize200}
\vspace*{0.2cm}
\begin{tabular}{|l|l|ccc|}
\hline 
\rule[-1ex]{0pt}{2.5ex} model & method  & $\alpha=10\%$ & $\alpha=5\%$ & $\alpha=1\%$ \\ 
\hline
FAR(1)  &  FSB & 0.107 & 0.051 & 0.009 \\ 
Gaussian  & NBB &  0.121 & 0.063& 0.008\\ 
$C=0.245$ &  Asymptotic & 0.113 & 0.062 & 0.010 \\ 
\hline
FAR(1)  &  FSB & 0.094 & 0.039 & 0.007 \\ 
Gaussian  &  NBB & 0.122 & 0.058 &0.012 \\ 
$C=0.49$ & Asymptotic & 0.122 & 0.058 & 0.011\\ 
\hline 
FAR(1)  &  FSB & 0.105 & 0.052 & 0.011 \\ 
Gaussian &   NBB & 0.083 & 0.034 & 0.006 \\ 
$C=-0.49$  &  Asymptotic & 0.081 & 0.037 & 0.005\\ 
\hline
FAR(1)  & FSB & 0.096 & 0.048  &  0.009 \\ 
non-Gaussian & NBB& 0.120 & 0.056& 0.010 \\ 
$C=0.49$ &  Asymptotic & 0.123 & 0.061& 0.010 \\ 
\hline
FMA(1) &  FSB &  0.095& 0.041 &0.006 \\ 
fast decay &  NBB & 0.098 & 0.047 & 0.012\\ 
of autocov. &  Asymptotic &  0.093 & 0.046 & 0.012 \\  
\hline
FMA(1) &  FSB & 0.101 & 0.051 & 0.009\\ 
slow decay &  NBB & 0.113 & 0.052 & 0.013\\ 
of autocov. &  Asymptotic & 0.110  & 0.059 & 0.013 \\  
\hline 
\end{tabular} 
\end{table}

To simulate the behavior under the alternative, we generate $n=200$ observations by
\begin{equation*}
Y_n(t)=\begin{cases}X_n(t) \ \  &\text{for }n\leq 100\\ X_n(t)+\mu \ \ &\text{for }n\geq 101\end{cases},
\end{equation*}
 where $\mu$ is chosen to be constant (not dependent on $t$)  with values $0.1$, $0.15$ or $0.3$ depending on the dependence structure. We adjusted the critical values such that the size under the null-hypothesis would be exactly the nominal level, so that we get the size-corrected power under the alternative. While the difference in the power of the three methods is not very pronounced, the other two methods lead to slightly higher power, see Figure \ref{fig1} for the AR(1)-process with $C=0.49$ and Figure \ref{fig1b} in the appendix for further results. As the test statistic used it the same and the only difference is the method to obtain critical values, it is not surprising that  the power of the tests  behaves similar. 

\begin{figure}
 \caption{Size-corrected empirical power for a FAR(1)-process with $C=0.49$ and Gaussian innovations, jump of size $\mu=0.15$ after $100$ of the $n=200$ observations (FSB = functional sieve bootstrap, NBB = non-overlapping block bootstrap, Asymptotic = method by \cite{AUE}).}\label{fig1}
\includegraphics[width=0.8\textwidth]{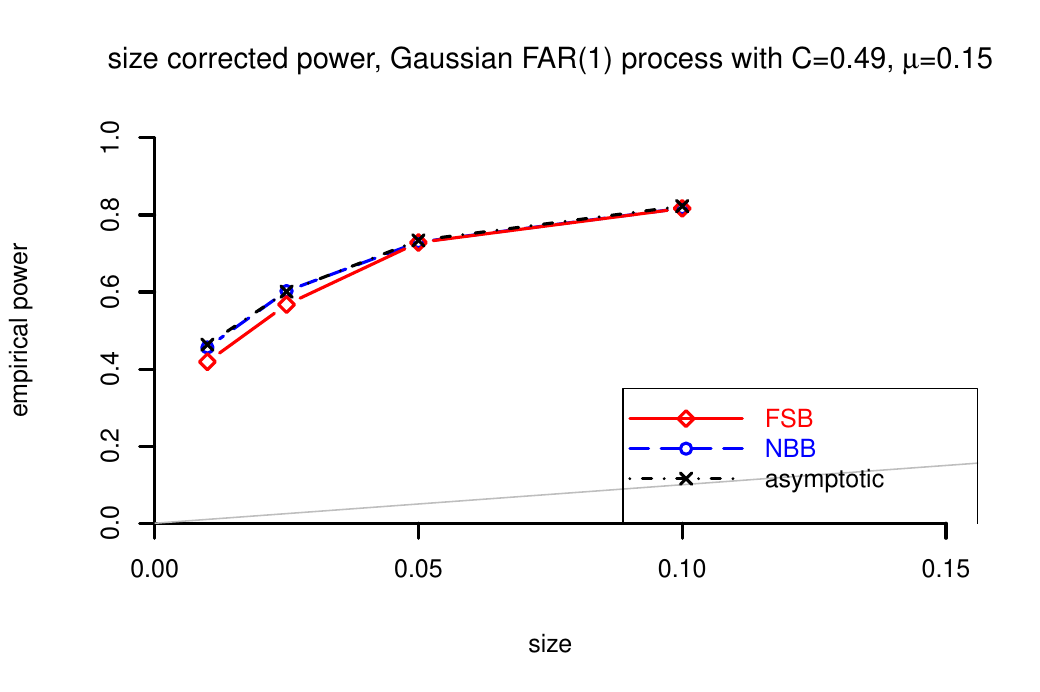} 
 \end{figure}

\bibliography{paparoditisetalarxiv4.bbl} 

\clearpage
%\newpage 

\begin{Large}
\noindent {\bf Supplementary Material} 
\end{Large}

\appendix
 
\section{Proofs}
We recall the following notation which will be used in the sequel:
\begin{itemize}
\item $\hat{v}_1,...,\hat{v}_m$  are the estimated eigenfunctions corresponding to the estimated eigenvalues $\hat{\lambda}_1 > \hat{\lambda}_2 > ... > \hat{\lambda}_m$ of the sample covariance operator 
\item $\hat{\xi}_t(m) = (\hat{\xi}_{j,t}, j=1,...,m )^\top $, where  $\hat{\xi}_{j,t}=\langle X_t, \hat{v}_j \rangle$, is the $m$-dimensional vector of estimated scores.
\item $\hat{X}_{t,m} = \sum_{j=1}^m \hat{\xi}_{j,t}\hat{v}_j$ and  $\hat{U}_{t,m} = X_t - \hat{X}_{t,m}$, $t=1,...,n$.
\item $U^\ast_{t,m}$ is drawn with replacement  from the set  \\$\{(\hat{U}_{t,m}-\frac{1}{n}\sum_{s=1}^n \hat{U}_{s,m}), t=1,...,n
\}$
\item$\hat{A}_{j,p}(m)$, $j=1,...,p$ are estimates of AR-matrices from $p$-th order  VAR-process fitted to the vector time series $\hat{\xi}_t$, $t=1,2 \ldots, n$.
\item residuals $\hat{\epsilon}_{t,p}(m) = \hat{\xi}_t(m) - \sum_{j=1}^p \hat{A}_{j,p}(m)\hat{\xi}_{t-j}(m)$, $t= p+1,p+2,...,n$
\item   $\xi_t^\ast =(\xi_{j,t}^\ast)_{j=1,...,m}$, $t=1,2, \ldots,n$, is a  $m$-dimensional, FSB generated  pseudo time series, where  $\xi_t^\ast = \sum_{j=1}^p \hat{A}_{j,p}(m) \xi_{t-j}^\ast + e_t^\ast$, $t=1,...,n$. $e^\ast_t$ is drawn iid from the centered residuals: $e^\ast_t=(\hat{e}_{I_t,p}- \frac{1}{n-p}\sum_{s=p+1}^n \hat{e}_{s,p})$. $I_1,....,I_n$ are the same independent and uniformly on $\{p+1,...,n\}$ distributed random variables as used for the construction of $e^\star_t$.
\item  $X_t^\ast = \sum_{j=1}^m \xi_{j,t}^\ast \hat{v}_j + U_{t,m}^\ast$, $t=1,2, \ldots,n$ is the FSB generated  functional time series.  
\end{itemize}

For simplicity  and if it is clear from the context, we avoid in the following the notation $ y(m)$ for a $m$-dimensional vector and simply  write $y$. \\

\noindent {\bf Proof of Theorem~\ref{th.boot}:} We will make use of the following two theorems due to  \cite{SER}, which we give here for ease of reference. 

\begin{thm}[Theorem A of \cite{SER}] \label{ThmA}
Let $(X_i)_{i \in \mathbb{N}}$ be a series of random variables, $\nu \geq 2$. Suppose it exists a function $g(F_{a,n})$ (depending on the joint distribution function $F_{a,n}$ of $X_{a+1},...,X_{a+n}$) satisfying
\begin{equation}
 g(F_{a,k})+ g(F_{a+k,l}) \leq g(F_{a,k+l}) \;\;\; \forall a\geq a_0, \; 1 \leq k \leq k+l \label{ThmAineq}  
 \end{equation}
such that $\E[\| S_{a,n}\| ^\nu ] \leq g^{\frac{1}{2}\nu}(F_{a,n})$, then
\[ \E (M_{a,n}^\nu) \leq \log_2(2n)^\nu g^{\frac{1}{2}\nu}(F_{a,n})  \]
where $S_{a,n}= \sum_{i=a+1}^{a+n}X_i$ and $M_{a,n}=\max\limits_{1\leq k \leq n} \| S_{a,k}\|$.
\end{thm}

\begin{thm}[Theorem B of \cite{SER}] \label{ThmB}
Let $\nu >2$ and use the same notation as in Theorem~\ref{ThmA}.  Suppose that  $ \E |S_{a,n}|^\nu \leq g^{\nu/2}(n)$, for all $ a \geq a_0$ and all $n\geq 1$, where $ g(n)$ is  nondecreasing, $ 2g(n) \leq g(2n)$, and $g(n)/g(n+1)\rightarrow 1$ as $ n \rightarrow \infty$. Then there exists a finite constant $K$ (which may depend on $\nu$, $g$ and the joint distributions of the $ X_t's$) such that 
\[ \E(M_{a,n}^\nu ) \leq K g^{\nu/2}(n).\] 
\end{thm}

Note that these results  were  formulated for real-valued random variables by  \cite{SER}, but the proofs carry over to normed spaces without any changes.

%For any $ m\in\mathbb{N}$ we assume that the $m$-dimensional vector of scores $ \xi_s=(\xi_{s,1}, \xi_{s,2}, \ldots, \xi_{s,m})^\top$, $ \xi_{j,s}=<X_s,v_j>$, obeys the VAR representation
%\begin{equation} \label{eq.xi-VAR} \xi_s = \sum_{j=1}^\infty A_j(m) \xi_{s-j} + e_s,
%\end{equation}
%where $ e_s \sim WN(0,\Sigma_e(m))$, is a $m$-dimensional white noise process with mean zero and covariance matrix $\Sigma_e(m)$. Let $  G_e$ be the distribution function of $e_s$ (which by the strict stationarity of $X_t$ does not depend on $s$).
%
Define the following fictitious processes:
\begin{itemize}
\item $\{ \tilde{\xi}_s, s \in \mathbb{Z} \}$ is a $m$-dimensional process which obeys   a  vector autoregressive representation as $\xi_s$, i.e., 
\begin{equation} \label{eq.tildexi}
\tilde{\xi}_s = \sum_{j=1}^\infty A_j(m)\tilde{\xi}_{s-j} + \varepsilon_s
\end{equation}
where the set  of   $m\times m$ coefficient matrices $\{A_j(m) , j\in\mathbb{N}\}$ is the same  as    in (\ref{eq.xi-VAR})  but  the innovations  $\varepsilon_s$  are  i.i.d. sequence with mean zero, variance $ \Sigma_e(m)$  and distribution function  $ G_e$.  That is,  
 in contrast to 
the innovations $e_t$ in (\ref{eq.xi-VAR}), the innovations $ \varepsilon_t $ in (\ref{eq.tildexi})   are i.i.d., which implies that  $\{\tilde\xi_s, s\in\mathbb{Z}\}$ is a linear, $m$-dimensional VAR($\infty)$ process. 
%obtained by iid resampling of centered residuals $\{(\epsilon_t-\bar{\epsilon}_n), t=1,...,n\}$ with $\epsilon_t = (\xi_t - \sum_{j=1}^\infty A_j(m) \xi_{t-j}) $
\item $\{\xi_s^+, s \in\mathbb{Z} \}$  is a $m$-dimensional process, where $\xi_t^+$ is generated as
\[\xi_t^+ = \sum_{j=1}^p \tilde{A}_{j,p}(m) \xi_{t-j}^+ + \epsilon_t^+ . \]
Here $\tilde{A}_{p,m}=(A_{1,p}(m),...,A_{p,p}(m))$ is as $\hat{A}_{p,m}$ but with regard to the true series $\xi_t$, $ t=1,2, \ldots, n$ and $\epsilon_t^+$ is obtained by resampling from centered residuals: $\epsilon_t^+=\tilde{\epsilon}_{I_t} -\bar{\tilde{\epsilon}}_n$  with $\tilde{\epsilon}_t = (\xi_t - \sum_{j=1}^p \tilde{A}_{j,p}(m)\xi_{t-j}$ and $ \bar{\tilde\epsilon}=(n-p)^{-1}\sum_{t=p+1}^n \tilde{\epsilon}_t$. $I_1,....,I_n$ are the same independent and uniformly on $\{p+1,...,n\}$ distributed random variables as used for the construction of $e^\star_t$.
\item $\{ \xi_s^\circ, s\in\mathbb{Z} \}$ is  a $\mathbb{R}^\mathbb{N}$ dimensional process  which satisfies the following condition: For each $m\in \mathbb{N}$, it holds true that 
the $m$-dimensional  vector process
$\{ \xi^\circ_t(m), t\in \mathbb{Z}\}$, where $ \xi^\circ_t(m)$   consists of the first $m$ components of the infinite dimensional vector $\xi^\circ_t$, coincides with the $m$-dimensional process $\{\tilde{\xi}_t,t\in\mathbb{Z}\}$ given in (\ref{eq.tildexi}).
%$, where $ \xi^\circ_t(m) $ consists  of the first $m$ components of  infinite dimensional vector $ \xi_t^\circ$.
%\[\I_{m,\infty} \cdot \xi_s^\circ = \tilde{\xi}_s, \text{ where } \I_{m,\infty}= \begin{pmatrix} 
%1 & 0 & 0 & \cdots & 0 & 0 & \cdots
%\\ 0 & 1 & 0  & \cdots & 0 & 0 & \cdots
%\\ 0 & 0 & 1 & \cdots & 0 & 0 & \cdots
%\\ \vdots & \vdots & \vdots & \ddots  & &
%\\ 0 & 0 & 0 & \cdots & 1 & 0 & \cdots
%\end{pmatrix}_{m\times \infty} \]
\end{itemize}
Define next, 
\begin{itemize}
\item[(i)] $Z^+_{n,m}(t) = \frac{1}{\sqrt{n}} \sum_{s=1}^{\floor*{nt}} \sum_{l=1}^m \xi_{l,s}^+ v_l$, with $\xi_s^+ = (\xi_{1,s}^+ , \cdots ,\xi_{m,s}^+)\Transp \in \mathbb{R}^m$,
\item[(ii)]  $\tilde{Z}_{n,m}(t) = \frac{1}{\sqrt{n}} \sum_{s=1}^{\floor*{nt}} \sum_{l=1}^m \tilde{\xi}_{l,s} v_l$, with $\tilde{\xi}_s = 
(\tilde{\xi}_{1,s}, \cdots,  \tilde{\xi}_{m,s})\Transp \in \mathbb{R}^m$,
\item[(iii)]  $Z_n^\circ(t) = \frac{1}{\sqrt{n}} \sum_{s=1}^{\floor*{nt}} \sum_{l=1}^\infty \xi_{l,s}^\circ v_l$, with $\xi_s^\circ = (\xi_{1,s}^\circ, \xi_{2,s}^\circ, \cdots)\Transp \in \mathbb{R}^\mathbb{N}$.
\end{itemize}
Using 
\begin{align*}
Z_{n,m}^\ast (t) = Z_n^\circ (t) + (\tilde{Z}_{n,m}(t) -  Z_n^\circ (t)) + (Z_{n,m}^+(t) - \tilde{Z}_{n,m}(t)) \\
+ (Z_{n,m}^\ast (t) - Z_{n,m}^+(t)),
\end{align*}
the assertion of the theorem  follows from the following  lemmas.\\

\begin{lem}\label{le.lemma1}
 Under the assumptions of Theorem~\ref{th.boot} it holds true,  as $n \rightarrow \infty$, that 
\[ \sup\limits_{t\in[0,1]} \Vert Z_{n,m}^\ast (t) - Z_{n,m}^+(t) \Vert \overset{P}{\rightarrow} 0.\] 
\end{lem}

\begin{lem}\label{le.lemma2}
 Under the assumptions of Theorem~\ref{th.boot} it holds true, it is possible (after enlarging the probability space if needed) to define copies $Z_{c,n,m}^{+}$ of $Z_{n,m}^+$ and $\tilde{Z}_{c,n,m}$ of $\tilde{Z}_{n,m}$, such that
\[\sup\limits_{t\in[0,1]} \Vert Z_{n,m}^+(t) - \tilde{Z}_{n,m}(t) \Vert \overset{P}{\rightarrow} 0\]
 as $n \rightarrow \infty$.
\end{lem}

\begin{lem} \label{le.lemma3}
 Under the assumptions of Theorem~\ref{th.boot} it holds true,  as $n \rightarrow \infty$, that 
\[ \sup\limits_{t\in[0,1]} \Vert \tilde{Z}_{n,m}(t) -  Z_n^\circ (t) \Vert \overset{P}{\rightarrow} 0.\]
\end{lem}

\begin{lem} \label{le.lemma4}
 Under the assumptions of Theorem~\ref{th.boot} it holds true as $n \rightarrow \infty$, that 
\[  (Z_n^\circ (t))_{t \in [0,1]} \Rightarrow W,\]
where $W$ is the  Brownian motion  given in Theorem~\ref{th.boot}.
\end{lem}

\vspace*{0.2cm}

\noindent{\bf Proof of Lemma~\ref{le.lemma1}:} \  Using the definition of $Z_{n,m}^\ast$ and $Z_{n,m}^+$, we write
\begin{align*}
&Z_{n,m}^\ast(t) - Z_{n,m}^+(t) = \frac{1}{\sqrt{n}} \sum_{s=1}^{\floor*{nt}} \sum_{l=1}^m  ( \xi_{l,s}^\ast \hat{v}_l + U_{s,m}^\ast - \xi_{l,s}^+ v_l) \\
&= \frac{1}{\sqrt{n}} \sum_{s=1}^{\floor*{nt}} \sum_{l=1}^m (\xi_{l,s}^\ast \hat{v}_l- \xi_{l,s}^+v_l \pm \xi_{l,s}^\ast v_l) + \frac{1}{\sqrt{n}} \sum_{s=1}^{\floor*{nt}} U_{s,m}^\ast \\
&= \frac{1}{\sqrt{n}} \sum_{s=1}^{\floor*{nt}} \sum_{l=1}^m  \xi_{l,s}^\ast(\hat{v}_l -v_l) + \frac{1}{\sqrt{n}} \sum_{s=1}^{\floor*{nt}} \sum_{l=1}^m (\xi_{l,s}^\ast - \xi_{l,s}^+)v_l + \frac{1}{\sqrt{n}} \sum_{s=1}^{\floor*{nt}} U_{s,m}^\ast \\
& =: V_{n,m}^\ast (t) + D_{n,m}^\ast (t) + R_{n,m}^\ast (t)
\end{align*}
We will show convergence to zero for $V_{n,m}^\ast, D_{n,m}^\ast$ and $R_{n,m}^\ast$ separately. For $V_{n,m}^\ast$ and $R_{n,m}^\ast$, we will use Theorem~\ref{ThmA}    combined with the results of \cite{PAP}, Lemma 6.8 and Lemma 6.6. 
%\begin{thm}[Theorem A of \cite{SER}] \label{ThmA}
%Let $(X_i)_{i \in \mathbb{N}}$ be a series of random variables, $\nu \geq 2$. Suppose it exists a function $g(F_{a,n})$ (depending on the joint distribution function of $X_{a+1},...,X_{a+n}$) satisfying
%\[ \tag{\#} g(F_{a,k})+ g(F_{a+k,l}) \leq g(F_{a,k+l}) \;\;\; \forall a\geq a_0, \; 1 \leq k \leq k+l \label{ThmAineq}  \]
%such that $\E[\| S_{a,n}\| ^\nu ] \leq g^{\frac{1}{2}\nu}(F_{a,n})$, then
%\[ \E[ M_{a,n}^\nu] \leq \log_2(2n)^\nu g^{\frac{1}{2}\nu}(F_{a,n})  \]
%where $S_{a,n}= \sum_{i=a+1}^{a+n}X_i$ and $M_{a,n}=\max\limits_{1\leq k \leq n} \| S_{a,k}\|$.
%\end{thm}
%Note that this theorem was formulated for real values random variables by \emph{Serfling - add citation}, but the proof and result carris over to over normed spaces without any changes.

Consider $V_{n,m}^\ast$.  Recall that $m$ and $p$ depend on $n$. We will sometimes write $m(n)$ and $p(n)$ in the following. First, we define the function $g$ from Theorem \ref{ThmA}, using calculations as in Lemma 6.8 of  \cite{PAP}.
\begin{align*}
 V_{n,m}^\ast(1) &= \frac{1}{\sqrt{n}} \sum_{s=1}^n \sum_{l=1}^{m} \xi_{l,s}^\ast (\hat{v}_l -v_l) \\
 \E[\Vert V_{n,m}^\ast(1) \Vert^2] &\leq \sum_{l=1}^{m} \E\Vert \hat{v}_l - v_l \Vert^2 \cdot \frac{1}{n}\sum_{r=1}^n \sum_{s=1}^n \Vert \Gamma_{r-s}^\ast \Vert_F\\
&\leq C \frac{1}{n} \sum_{j=1}^{m}\alpha_j^{-2}\frac{1}{n}\sum_{r=1}^n \sum_{s=1}^n \Vert \Gamma_{r-s}^\ast \Vert_F
\end{align*} 
where $\Gamma_{r-s}^\ast= \E[\xi_r^\ast \xi_s^{\ast^T}]$ the autocovariance matrix function of the process $\{\xi_t^\ast\}$. 
For the last inequality, we used from \cite{PAP}, that,
\[\sum_{l=1}^{m} \E\Vert \hat{v}_l -v_l \Vert^2 \leq \mathcal{O}_P(\frac{1}{n} \sum_{j=1}^{m}\alpha_j^{-2} ). \]
With the same arguments we get for any $a\geq 0$, that, 
\begin{align*}
\E\Vert \sum_{s=a+1}^{a+n} \sum_{l=1}^{m} \xi_{l,s}^\ast (\hat{v}_l -v_l) \Vert^2  \leq C \Big(\frac{1}{n} \sum_{j=1}^{m}\alpha_j^{-2}\Big)\sum_{r=a+1}^{a+n} \sum_{s=a+1}^{a+n} \Vert \Gamma_{r-s}^\ast  \Vert =: g^{\frac{1}{2}\nu}(F_{a,n})
\end{align*}
Next, check (\ref{ThmAineq}) for $\nu=2$:
\begin{align*}
&g(F_{a,k}) +g(F_{a+k,l}) \\
 =&C \frac{1}{n} \sum_{j=1}^{m}\alpha_j^{-2}\sum_{r=a+1}^{a+k}\sum_{s=a+1}^{a+k}\Vert \Gamma_{r-s}^\ast \Vert  + C \frac{1}{n} \sum_{j=1}^{m}\alpha_j^{-2}\sum_{r=(a+k)+1}^{(a+k)+l}\sum_{s=(a+k)+1}^{(a+k)+l}\Vert \Gamma_{r-s}^\ast \Vert \displaybreak[0]\\
&= C \frac{1}{n} \sum_{j=1}^{m}\alpha_j^{-2}\Big(\sum_{r=a+1}^{a+k}\sum_{s=a+1}^{a+k}\Vert \Gamma_{r-s}^\ast \Vert  +\sum_{r=(a+k)+1}^{(a+k)+l}\sum_{s=(a+k)+1}^{(a+k)+l}\Vert \Gamma_{r-s}^\ast \Vert \Big)  \\
& \leq   C \frac{1}{n} \sum_{j=1}^{m}\alpha_j^{-2}\sum_{r=a+1}^{a+k+l}\sum_{s=a+1}^{a+k+l} \Vert \Gamma_{r-s}^\ast \Vert  = g(F_{a,k+l})
\end{align*}
By Theorem \ref{ThmA}, 
\begin{align*}
\E[ \max\limits_{1\leq k \leq n}\Vert \sum_{s=1}^k \sum_{l=1}^{m(n)} \xi_{l,s}^\ast (\hat{v}_l -v_l) \Vert^2] \leq \log_2(2n)^2g(F_0,n) \\
= \log_2(2n)^2 C \frac{1}{n} \sum_{j=1}^{m}\alpha_j^{-2}\cdot \sum_{r=1}^n \sum_{s=1}^n \Vert \Gamma_{r-s}^\ast \Vert,
\end{align*}
that is, 
\begin{align*}
 & \E[\sup\limits_{t\in [0,1]}\Vert  V_{n,m}^\ast(t) \Vert^2] = \E[ \max\limits_{1\leq k\leq n}\Vert \frac{1}{n^{1/2}} \sum_{s=1}^k \sum_{l=1}^{m(n)}\xi_{l,s}^\ast (\hat{v}_l -v_l) \Vert ^2] \\
& = \frac{1}{n}\E[ \max\limits_{1\leq k \leq n} \Vert\sum_{s=1}^k \sum_{l=1}^{m(n)} \xi_{l,s}^\ast (\hat{v}_l -v_l) \Vert^2] \\
& \leq \frac{1}{n} \log_2(2n)^2 C \frac{1}{n} \sum_{j=1}^{m}\alpha_j^{-2}\sum_{r=1}^n \sum_{s=1}^n \Vert \Gamma_{r-s}^\ast \Vert \\
& \leq \log_2(2n)^2 \mathcal{O}_P(\frac{1}{n}\sum_{j=1}^{m(n)} \alpha_j^{-2}) = \frac{1}{n^{1/2}} \log_2(2n)^2 \mathcal{O}_P(\frac{1}{n^{1/2}}\sum_{j=1}^{m(n)} \alpha_j^{-2}),
\end{align*}
and this converges to zero for $n \to \infty$, because we have by our assumptions that $\frac{1}{n}\sum_{r=1}^n \sum_{s=1}^n \Vert \Gamma_{r-s}^\ast \Vert \leq \mathcal{O}_P(1)$ and $1/\sqrt{n}\sum_{j=1}^{m(n)} \alpha_j^{-2}=O_P(1)$ .
%\emph{add assumptions on $alpha_j$ to this article XXXX }. Convergence in probability follows then by Markow's inequality. \\

Consider $R_{n,m}^\ast$.
We proceed similar as for $V_{n,m}^\ast$ in order  to define $g$.  In particular, we have 
\begin{align*}
&R_{n,m}^\ast(1) = \frac{1}{n^{1/2}} \sum_{s=1}^n U_{s,m}^\ast, \ \ \ \mbox{and,}\\
&\E\Vert R_{n,m}^\ast(1) \Vert^2 \leq \frac{2}{n}\sum_{s=1}^n \Vert \hat{U}_{s,m} \Vert^2 + 2 \Vert \bar{\hat{U}}_n \Vert^2,
\end{align*}
by the definition of $U_{s,m}^\ast$ in Step 3 of the bootstrap algorithm.  Then,  
\begin{align*}
%\Rightarrow  
\E\Vert \sum_{s=1}^n U_{s,m}^\ast \Vert^2 & \leq 2 \sum_{s=1}^n \Vert \hat{U}_{s,m} \Vert^2 + 2n \Vert \bar{\hat{U}}_n \Vert ^2 = 2 \sum_{s=1}^n \Vert \hat{U}_{s,m} \Vert^2 + 2n \Vert \frac{1}{n} \sum_{s=1}^n \hat{U}_{s,m} \Vert^2\displaybreak[0] \\
& \leq  2 \sum_{s=1}^n \Vert \hat{U}_{s,m} \Vert^2 + 2n(\frac{1}{n} \Vert  \sum_{s=1}^n \hat{U}_{s,m} \Vert^2 )= 4 \sum_{s=1}^n \Vert \hat{U}_{s,m} \Vert^2 \\
& \leq 16 n \Vert \hat{C}_0 \Vert_{HS} \Big( (\sum_{j=1}^{m} \Vert \hat{v}_j -v_j \Vert)^2 + \sum_{j=1}^{m}\Vert \hat{v}_j -v_j \Vert^2 \Big),
\end{align*}
as in Lemma 6.6 of \cite{PAP}. % where $\hat{C}_0 = \frac{1}{n} \sum_{t=1}^n (X_t-\bar{X}_n)\otimes (X_t-\bar{X}_n)$ is the lag zero, sample covariance operator. 
So, it holds for any $a\geq 0$ and for $\nu =2$ that
\[\E\Vert \sum_{s=a+1}^{a+n} U_{s,m}^\ast \Vert^2  \leq 16n \Vert \hat{C}_0 \Vert_{HS}\Big( (\sum_{j=1}^{m} \Vert \hat{v}_j -v_j \Vert)^2 + \sum_{j=1}^{m}\Vert \hat{v}_j -v_j \Vert^2 \Big).\]
Next, check (\ref{ThmAineq}). We have, 
\begin{align*}
g(F_{a,k}) + g(F_{a+k,l}) =&  16k \Vert \hat{C}_0 \Vert_{HS}\Big( (\sum_{j=1}^{m} \Vert \hat{v}_j -v_j \Vert)^2 + \sum_{j=1}^{m}\Vert \hat{v}_j -v_j \Vert^2 \Big) \\
& + 16l \Vert \hat{C}_0 \Vert_{HS}\Big( (\sum_{j=1}^{m} \Vert \hat{v}_j -v_j \Vert)^2 + \sum_{j=1}^{m}\Vert \hat{v}_j -v_j \Vert^2 \Big)\\
 =& 16(k+l)\Vert \hat{C}_0 \Vert_{HS} \Big( (\sum_{j=1}^{m} \Vert \hat{v}_j -v_j \Vert)^2 + \sum_{j=1}^{m}\Vert \hat{v}_j -v_j \Vert^2 \Big).
\end{align*}
From the proof of Lemma 6.6 of \cite{PAP}, we have that
\begin{equation*}
\Vert \hat{C}_0 \Vert_{HS}  \Big( (\sum_{j=1}^{m} \Vert \hat{v}_j -v_j \Vert)^2 + \sum_{j=1}^{m}\Vert \hat{v}_j -v_j \Vert^2 \Big) \leq \mathcal{O}_p(\frac{1}{n^{1/2}} \sum_{j=1}^{m} \alpha_j^{-2}),
\end{equation*}
so we can conclude with the help of Theorem \ref{ThmA}  that
\begin{align*}
& \E\Vert \sup\limits_{t \in [0,1]} R_{n,m}^\ast (t) \Vert^2 ]  = \E\Vert \max\limits_{1\leq k\leq n} \frac{1}{n^{1/2}}\sum_{s=1}^k U_{s,m}^\ast \Vert^2 \\
 =& \frac{1}{n}\E\Vert \max\limits_{1\leq k\leq n} \sum_{s=1}^k U_{s,m}^\ast \Vert^2\leq \log_2(2n)^2 g(F_{0,n}) \\
= &\frac{1}{n}\log_2(2n)^2 16 n \Vert \hat{C}_0 \Vert_{HS}  \Big( (\sum_{j=1}^{m} \Vert \hat{v}_j -v_j \Vert)^2 + \sum_{j=1}^{m}\Vert \hat{v}_j -v_j \Vert^2 \Big) \\
\leq& \log_2(2n)^2 \mathcal{O}_p(\frac{1}{n^{1/2}} \sum_{j=1}^{m(n)} \alpha_j^{-2})  = \frac{1}{n^{1/4}} \log_2(2n)^2 \mathcal{O}_P(\frac{1}{n^{1/4}} \sum_{j=1}^{m(n)} \alpha_j^{-2}),
\end{align*}
and this is independent of $t$ and converges to zero for $n \to \infty$. 
%(Assumptions needed XXXX ?)\\

Consider next $D_{n,m}^\ast$.  To handle this term, we 
 proceed differently. We will show that for arbitrary $t\in [0,1]$, respectively $1\leq k\leq n$, it holds that, 
\[ \frac{1}{n^{1/2}} \sum_{s=1}^{\floor*{nt}} \sum_{l=1}^{m} (\xi_{l,s}^\ast - \xi_{l,s}^+)v_l = \frac{1}{n^{1/2}} \sum_{s=1}^{k} \sum_{l=1}^{m} (\xi_{l,s}^\ast - \xi_{l,s}^+)v_l    \]
converges to zero in probability.  
To do so, we will follow the lines of the proof of Lemma 6.7 (\cite{PAP}). We have, 
\begin{align*}
\E\Vert D_{n,m}^\ast (t) \Vert^2 = \frac{1}{n} \sum_{r,s=1}^{\floor*{nt}} \sum_{l=1}^{m} \I_l\Transp \E[\xi_r^\ast (\xi_s^\ast - \xi_s^+)\Transp] \I_l+\frac{1}{n} \sum_{r,s=1}^{\floor*{nt}} \sum_{l=1}^{m} \I_l\Transp \E[\xi_r^+ (\xi_s^+ - \xi_s^\ast)\Transp] \I_l \\
=: D_{n,m}^{(1)}(\floor*{nt}) +  D_{n,m}^{(2)}(\floor*{nt})
\end{align*}
For simpler notation, we write $k=\floor*{nt}$, $1 \leq k \leq n$. Starting with $D_{n,m}^{(1)}$, we upper bound the expression:
\begin{align}
& D_{n,m}^{(1)}(k) = \frac{1}{n} \sum_{r,s=1}^k \sum_{l=1}^{m} \I_l\Transp \E[\xi_r^\ast(\xi_s^\ast -\xi_s^+)\Transp] \I_l \nonumber\\
& = \frac{1}{n} \sum_{r,s=1}^k \sum_{j=1}^{m} \sum_{l=0}^\infty \I_j\Transp \hat{\Psi}_{l,p}(m) \Sigma_{\epsilon}^\ast (m) \Big(\hat{\Psi}_{l+s-r,p}(m) - \tilde{\Psi}_{l+s-r,p}(m)\Big)\Transp \I_j \label{1}\\
& + \frac{1}{n}\sum_{r,s=1}^k\sum_{j=1}^{m}\sum_{l=0}^\infty \I_j\Transp \hat{\Psi}_{l,p}(m) \E[\epsilon_{r,p}^\ast(m)\Big(\epsilon_{t,p}^\ast(m) - \epsilon_{t,p}^+(m)\Big)]\tilde{\Psi}_{l+s-r,p}(m)\Transp \I_j \label{2}
\end{align}
Here, $\Sigma^\ast_{\epsilon}\delta_{t,s}= \E[\epsilon_{t,p}^\ast \epsilon_{s,p}^{\ast\Transp}]$ and $\tilde{\Psi}_{j,p}(m)$,  respectively,  $\hat{\Psi}_{j,p}(m)$ $j=1,2,...$, are the coefficient matrices of  the power series  $\hat{A}_{p,m}^{-1}(z)$,
 respectively,  $\tilde{A}_{p,m}^{-1}(z)$,  $|z|\leq 1$, with $\hat{\Psi}_{0,p}(m) = \tilde{\Psi}_{0,p}(m) = \I_m$. \\
We will handle terms (\ref{1}) and (\ref{2}) separately. 
\begin{multline*}
 \Vert (\ref{1})\Vert_F\\
\leq \Vert \Sigma^\ast_{\epsilon}(m)\Vert_F \sum_{l=0}^\infty \Vert \sum_{j=1}^{m} \I_j\Transp \hat{\Psi}_{l,p}(m) \Vert_F \frac{1}{n}\sum_{r,s=1}^k \Vert \sum_{j=1}^{m} \I_j\Transp \big(\hat{\Psi}_{l+s-r,p}(m) - \tilde{\Psi}_{l+s-r,p}(m) \big)\Vert_F \\
\leq \Vert \Sigma_{\epsilon}(m)\Vert_F \sum_{l=0}^\infty \Vert \sum_{j=1}^{m} \I_j\Transp \hat{\Psi}_{l,p}(m) \Vert_F \sum_{l=0}^\infty \Vert \sum_{j=1}^{m} \I_j\Transp \hat{\Psi}_{l,p}(m) - \tilde{\Psi}_{l,p}(m) \Vert_F \leq \mathcal{O}_P(1) 
\end{multline*}
since Lemma 6.1 and 6.5 of \cite{PAP} hold uniformly in m (and p).
\begin{multline*}
 \Vert (\ref{2}) \Vert_F = \Vert \frac{1}{n} \sum_{r,s=1}^k \sum_{j=1}^{m} \sum_{l=0}^\infty \I_j\Transp \hat{\Psi}_{l,p}(\hspace{-1pt} m \hspace{-1pt} ) \E[\epsilon_{r,p}^\ast(\hspace{-1pt} m \hspace{-1pt} ) \big(\epsilon_{t,p}^\ast(\hspace{-1pt} m \hspace{-1pt} ) - \epsilon_{t,p}^+(\hspace{-1pt} m \hspace{-1pt} )\big)] \tilde{\Psi}_{l+s-r,p}( \hspace{-1pt} m \hspace{-1pt} ) \Vert_F\\
 \leq \sqrt{\E\Vert \epsilon_{r,p}^\ast (\hspace{-1pt} m \hspace{-1pt} ) \Vert^2\E\Vert \epsilon_{t,p}^\ast(\hspace{-1pt} m \hspace{-1pt} ) \!- \!\epsilon_{t,p}^+(\hspace{-1pt} m \hspace{-1pt} ) \Vert^2} \cdot \underbrace{ \sum_{l=0}^\infty \Vert \sum_{j=1}^{m} \I_j\Transp \hat{\Psi}_{l,p}(\hspace{-1pt} m \hspace{-1pt} ) \Vert}_{\leq \mathcal{O}_P(1)} \cdot \underbrace{\sum_{l=0}^\infty \Vert \sum_{j=1}^{\hspace{-1pt} m \hspace{-1pt} } \I_j\Transp \tilde{\Psi}_{l,p}(m) \Vert}_{\leq \mathcal{O}_P(1)}
\end{multline*}
uniformly in $m$ (and $p$). We will now show that $\E[\Vert \epsilon_{r,p}^\ast(m) - \epsilon_{r,p}^+(m) \Vert^2]$ converges to zero in probability. 
\begin{align}
&\E[\Vert \epsilon_{r,p}^\ast(m) - \epsilon_{r,p}^+(m) \Vert^2] \nonumber \\
 \leq& \frac{2}{n-p}\sum_{r=p+1}^n \Vert \hat{\epsilon}_{r,p}(m)-\tilde{\epsilon}_{r,p}(m)\Vert^2 + 4 \big(\Vert \bar{\hat{\epsilon}}_n(m) \Vert^2 + \Vert \bar{\tilde{\epsilon}}_n(m) \Vert^2\big)\nonumber \\
 \leq& \frac{4}{n-p} \sum_{r=p+1}^n \Vert \hat{\xi_r}(m) - \xi_r(m) \Vert^2 \label{a}\\
& + \frac{4}{n-p}\sum_{r=p+1}^n \Vert \sum_{j=1}^{p} \hat{A}_{j,p}(m) \hat{\xi}_{r-j}(m) - \tilde{A}_{j,p}(m)\xi_{r-j}(m) \Vert^2 \label{b}\\
& + 4 \big(\Vert \bar{\hat{\epsilon}}_n(m) \Vert^2 + \Vert \bar{\tilde{\epsilon}}_n(m) \Vert^2\big) \label{c}
\end{align}
For (\ref{a}), note that
\begin{align*}
\frac{1}{n-p} \sum_{r=p+1}^n \Vert \hat{\xi}_r(m) -\xi_r(m) \Vert^2 \leq \frac{1}{n-p} \sum_{r=p+1}^n \Vert X_r\Vert^2 \sum_{j=1}^{m}\Vert \hat{v}_j-v_j \Vert^2 \\
 = \mathcal{O}_P(\frac{1}{n}\sum_{j=1}^{m(n)} \alpha_j^{-2})\overset{n \to\infty}{\rightarrow} 0
\end{align*}
Next, 
\begin{align*}
(\ref{b})&  \leq 2 \sum_{j=1}^{p} \Vert \hat{A}_{j,p}(m) \Vert^2 \frac{1}{n-p} \sum_{r=p+1}^n \Vert \hat{\xi}_{r-j}(m) - \xi_{r-j}(m) \Vert \\
& + 2 \sum_{j=1}^{p}\Vert \hat{A}_{j,p}(m) - \tilde{A}_{j,p}(m) \Vert^2 \frac{1}{n-p} \sum_{r=p+1}^n \Vert \xi_{r-j}(m)\Vert^2 \\
& \leq \mathcal{O}_P(1) \cdot \mathcal{O}_P(\frac{1}{n}\sum_{j=1}^{m(n)}\alpha_j^{-2}) \\
 & + \mathcal{O}_P\Big( (p(n)\lambda_{m(n)}^{-1}\sqrt{m(n)}+ p(n)^2)^2\sqrt{\frac{1}{n}\sum_{j=1}^{m(n)}\alpha_j^{-2}}\Big)\cdot \mathcal{O}_P(\frac{m(n)}{n-p(n)}) \\
 & \leq \mathcal{O}_P(\frac{1}{n}\sum_{j=1}^{m(n)}\alpha_j^{-2}) + \mathcal{O}_P \Big( \lambda_{m(n)}^{-2} \frac{1}{n} m(n)p(n) \sum_{j=1}^{m(n)}\alpha_j^{-2} \Big) \overset{n \to\infty}{\rightarrow} 0
\end{align*}
by Lemma 6.1 and 6.3 (\cite{PAP}). And finally, for the last part
\begin{align*}
\Vert \bar{\hat{\epsilon}}_{n}(m) \Vert^2  & \leq 2 \Vert \frac{1}{n-p} \sum_{r=p+1}^n \hat{\xi}_r \Vert^2 + 2\Vert \sum_{j=1}^{p} \hat{A}_{j,p}(m) \frac{1}{n-p} \sum_{r=p+1}^n \hat{\xi}_{r-j} \Vert^2 \\
& = \mathcal{O}_P\big(\frac{m(n)}{n-p(n)} + \frac{1}{n}\sum_{j=1}^{m(n)} \alpha_j^{-2}\big)  \overset{n \to\infty}{\rightarrow} 0  
\end{align*}
as in the proof of Lemma 6.5 (\cite{PAP}).  That $ \Vert \bar{\tilde{\epsilon}}_n(m) \Vert^2$, convergence to zero can be shown similarly. Thus, we get that $(\ref{c}) \rightarrow 0$ as $n \to\infty$. \\
Combining the results for (\ref{a}), (\ref{b}) and (\ref{c}), we achieve that $\Vert (2) \Vert \rightarrow 0$ independently of $k$ and thus it follows that $ D_{n,m}^{(1)}(k) \rightarrow 0$ in probability for arbitrary $1 \leq k\leq n$. The convergence of $D_{n,m}^{(2)}(k)$ can be shown in a similar way, which then proves the desired convergence of $D_{n,m}^\ast$. \\
%Combining all the above results together with Markov's inequality then proves Lemma 1.

\vspace*{0.2cm}

\noindent{\bf Proof of Lemma~\ref{le.lemma2}:} Recall that $ Z_{n,m}^+$ is based on $\xi^+_s$ and $\tilde{Z}_{n,m}$ on $\tilde{\xi}_s$, where
\begin{align*}
\xi_s^+ = 
%\sum_{j=1}^{p} \tilde{A}_{j,p}(m) \xi_{s-j}^+ + \epsilon_{t,p}^+ =  
 \sum_{j=0}^\infty \tilde{\Psi}_{j,p}(m) \epsilon_{s-j}^+ \ \ \mbox{and}  \ \ 
\tilde{\xi}_s  \
%sum_{j=1}^\infty A_j(m) \tilde{\xi}_{s-j} + \varepsilon_s   
= \sum_{j=0}^\infty \Psi_j(m) \varepsilon_{s-j},
\end{align*}
with $m\times m $ coefficient matrices $\tilde{\Psi}_{j,p}$ and $\Psi_j$ in the power series expansion  of $\tilde{A}_{p,m}^{-1}(z)=(I_m-\sum_{j=1}^p \tilde{A}_{j,p}(m) z^j)^{-1} $ and $
A_m^{-1}(z)=(I_m-\sum_{j=1}^\infty   A_{j,p}(m) z^j)^{-1} $, $|z|\leq 1$,  respectively, and $\tilde{\Psi}_{0,p}=\Psi_0=I_m$.  We write
\begin{align} \label{eq.decD}
\xi_s^+-\tilde{\xi}_s & = \sum_{j=0}^{\infty} \big( \tilde{\Psi}_{j,p}(m) -\Psi_j(m)\big)\epsilon_{s-j}^+ 
 + \sum_{j=0}^{\infty} \Psi_j(m) \big( \epsilon_{s-j}^+ - \widetilde{\varepsilon}_{s-j}\big) \nonumber  \\
&\ \ \ \  + \sum_{j=0}^{\infty} \Psi_j(m) \big( \widetilde{\varepsilon}_{s-j} - \varepsilon_{s-j}\big)  
\end{align}
 where   $ \tilde{\varepsilon}_r =e_{I_r}$ is  a pseudo random variable generated by i.i.d. resampling from  
the  centered set of  $n-p$ random variables $ e_{p+1}, e_{p+2},  \ldots, e_n$, where $e_s=\xi_s-\sum_{j=1}^\infty A_{j}(m) \xi_{s-j}$,  also see  (\ref{eq.xi-VAR}).
Using (\ref{eq.decD}) and  $k=\floor*{nt}$, we get 
\begin{align}
\frac{1}{\sqrt{n}}\sum_{s=1}^{\floor*{nt}}\sum_{l=1}^m \big(\xi_{l,s}^+-\tilde{\xi}_{l,s}\big)v_l & = \frac{1}{\sqrt{n}}\sum_{s=1}^{k}\sum_{l=1}^m I_l\sum_{j=0}^\infty \big(  \tilde{\Psi}_{j,p}(m) -\Psi_j(m)\big)\epsilon_{s-j}^+v_l\nonumber \\
& \ \ \ \   + 
 \frac{1}{\sqrt{n}}\sum_{s=1}^{k}\sum_{l=1}^m I_l\sum_{j=0}^\infty \Psi_j(m)\big(\epsilon_{s-j}^+ -\tilde{\varepsilon}_{s-j}\big) v_l \nonumber \\
 & \ \ \ \ +  \frac{1}{\sqrt{n}}\sum_{s=1}^{k}\sum_{l=1}^m I_l\sum_{j=0}^\infty \Psi_j(m)\big(\tilde{\varepsilon}_{s-j} - \varepsilon_{s-j} \big) v_l \nonumber \\
& = \tilde{D}^{(1)}_{n,m} (k)+  \tilde{D}^{(2)}_{n,m}(k)  + \tilde{D}^{(3)}_{n,m}(k),
\end{align}
with an obvious notation for $ \tilde{D}^{(i)}_{n,m}(k)$, $i=1,2,3$. Then
\begin{align}
\E\|D_{n,m}^{(1)}(k)\|^2  & =  
%\frac{1}{n} \sum_{r,s=1}^{k} \sum_{j=1}^{m} \I_j\Transp \E[\xi_r^+(\xi_s^+-\tilde{\xi}_s)\Transp]\I_j \nonumber\\
 \frac{1}{n} \sum_{r,s=1}^{k} \sum_{j=1}^{m} \sum_{l=0}^\infty \I_j\Transp \tilde{\Psi}_l(m) \Sigma_{\epsilon,p}^+(m) \big( \tilde{\Psi}_{l+s-r,p}(m) -\Psi_{l+s-r,p}(m) \big)\Transp \I_j \label{L2.1} \\
&-  \frac{1}{n} \sum_{r,s=1}^{k} \sum_{j=1}^{m} \sum_{l=0}^\infty \I_j\Transp \Psi_l(m) \Sigma_{\epsilon,p}^+(m) \big( \tilde{\Psi}_{l+s-r,p}(m) -\Psi_{l+s-r,p}(m) \big)\Transp \I_j \label{L2.2} 
%& \;\;\; + \frac{1}{n} \sum_{r,s=1}^{k} \sum_{j=1}^{m} \sum_{l=0}^\infty \I_j\Transp \tilde{\Psi}_l(m) \E[\epsilon_{r,p}^+(m) \big(\epsilon_{r,p}^+(m)-\varepsilon_r \big)\top] \Psi_{l+s-r}(m)\Transp \I_j \label{L2.2}
\end{align}
For  (\ref{L2.1}) we have by setting $ \Psi_{j+s}=\tilde{\Psi}_{j+s}=0$ for $ j+s<0$, that 
%. It is $\Sigma_{\epsilon,p}^+\delta_{t,s}=\E[\epsilon_{t,p}^+ \epsilon_{s,p}^{+\Transp} ]$.
\begin{align*}
 & \Vert (\ref{L2.1}) \Vert_F  \leq \Vert\Sigma_{\epsilon}^+(m)\Vert_F \sum_{l=0}^\infty \Vert \sum_{j=1}^{m} \I_j \tilde{\Psi}_l(m) \Vert_F \\
& \ \ \ \ \times \frac{1}{n}\sum_{r,s=1}^k \Vert \sum_{j=1}^{m} \I_j\Transp \big( \tilde{\Psi}_{l+s-r,p}(m) - \Psi_{l+s-r}(m) \big) \Vert_F \\
& \leq \Vert\Sigma_{\epsilon}^+(\hspace{-1pt} m \hspace{-1pt})\Vert_F  \sum_{l=0}^\infty \Vert \sum_{j=1}^{m} \I^\top_j \tilde{\Psi}_l(\hspace{-1pt} m \hspace{-1pt}) \Vert_F 
\sum_{s=-k+1}^{k-1}\frac{k-|s|}{n}\|\sum_{l=1}^m \I_l^\top (\tilde{\Psi}_{j+s}(\hspace{-1pt} m \hspace{-1pt})- \Psi_{j+s}(\hspace{-1pt} m \hspace{-1pt}))\|_F\\
& \leq 2  \Vert\Sigma_{\epsilon}^+(m)\Vert_F  \sum_{l=0}^\infty \Vert \sum_{j=1}^{m} \I_j \tilde{\Psi}_l(m) \Vert_F \sum_{l=0}^\infty \Vert \sum_{j=1}^{m} \I_j\Transp \big( \tilde{\Psi}_{l}(m) - \Psi_{l}(m) \big) \Vert_F \\
&= \mathcal{O}_P(1)o_P(1),
\end{align*}
 by Lemma 6.1 and 6.5 of (\cite{PAP}). By the same arguments it follows that (\ref{L2.2}) is $o_P(1)$, too.  
 
 For  $ D_{n,m}^{(2)}(k)$ we have
 \begin{multline} \label{eq.D2}
 \E\|D_{n,m}^{(2)}(k)\|^2 \\
 = \frac{1}{n} \sum_{r,s=1}^{k} \sum_{j=1}^{m} \sum_{l=0}^\infty \I_j\Transp \Psi_l(m) \E\big(\epsilon_{r,p}^+(m)  -\tilde{\varepsilon}_r\big)\big(\epsilon_{r,p}^+(m) -\tilde{\varepsilon}_r\big)^\top   \Psi_{l+s-r}(m)\Transp \I_j  
 \end{multline} 
Observe that 
\begin{multline*}\E\big(\epsilon_{r,p}^+(m)  -\tilde{\varepsilon}_r\big)\big(\epsilon_{r,p}^+(m) -\tilde{\varepsilon}_r\big)^\top \\
  = \Sigma^+_{\epsilon}(m) - 2 \E[\epsilon_{r,p}^+(m)\tilde{\varepsilon}_r ^\top] + 
 \frac{1}{n-p} \sum_{t=p+1}^n (e_t-\bar{e}) (e_t-\bar{e})^\top.
 \end{multline*}
 Now
 \begin{align*}
\E[\epsilon_{r,p}^+(m)\tilde{\varepsilon}_r ^\top] & = \frac{1}{n-p} \sum_{t=p+1}^n (\tilde{\epsilon}_t -\bar{\tilde{\epsilon}} ) (e_t-\bar{e})^\top,
\end{align*}
and $ \|\bar{\tilde{\epsilon}}\|\stackrel{P}{\rightarrow} 0$,  $ \|\bar{e}\|\stackrel{P}{\rightarrow} 0$, as in the proof of Lemma~\ref{le.lemma1}, while,
\begin{align*}
\frac{1}{n-p} \sum_{t=p+1}^n \tilde{\epsilon}_t e_t^\top & =  \frac{1}{n-p}\sum_{t=p+1}^n e_te_t^\top+ \frac{1}{n-p} \sum_{t=p+1}^n\sum_{j=1}^p(\tilde{A}_{j,p}(m)-A_j(m))\xi_{t-j}e^\top_t\\
& \ \ \ \ + \frac{1}{n-p} \sum_{t=p+1}^n\sum_{j=p+1}^\infty A_j(m))\xi_{t-j}e^\top_t\\
& = E_{1,n} + E_{2,n} + E_{3,n},
\end{align*}
with an obvious notation for $E_{i,n}$, $i=1,2,3$. Observe that $ \|E_{1,n} - \Sigma_e(m)\|_F\stackrel{P}{\rightarrow} 0$, while
 \begin{align*}
 \|E_{2,n}\|  & \leq \frac{1}{n-p}\sum_{r=p+1}^n \sqrt{\sum_{j=1}^p \Vert \tilde{A}_{j,p}(m) -A_{j,p}(m) \Vert_F^2} \sqrt{\sum_{j=1}^p \Vert \xi_{r-j}e^\top_t\Vert^2} \\
 & \leq \mathcal{O}_{P}(m^{-3} p^{-3/2}) =o(1),\\
 \|E_{3,n}\| & \leq  \mathcal{O}(mp^{-1}\sum_{j=p+1}^\infty j \|A_{j}(m)\|_F =o(1),  
  \end{align*}
  as in the proof of Lemma 6.4 in \cite{PAP}. Hence 
  \[  \|\E\big(\epsilon_{r,p}^+(m)  -\tilde{\varepsilon}_r\big)\big(\epsilon_{r,p}^+(m) -\tilde{\varepsilon}_r\big)^\top \|_F =o_P(1),\]
from which we conclude  using   
 \begin{align*}
  \|(\ref{eq.D2})\Vert_F &  \leq  \Vert \E\big(\epsilon_{r,p}^+(m)  -\tilde{\varepsilon}_r\big)\big(\epsilon_{r,p}^+(m) -\tilde{\varepsilon}_r\big)^\top  \Vert_F 
 \Big( \underbrace{\sum_{l=0}^\infty \Vert \sum_{j=1}^{m} \I_j\Transp \Psi_l(m) \Vert}_{\leq \mathcal{O}_(1)} \Big)^2,
  \end{align*}
 by Lemma 6.1 and 6.5 of \cite{PAP}, that $ D_{n,m}^{(2)}(k)=o_P(1)$.
 
 Consider next  $D_{n,m}^{(3)}(k)$. We will show that there are copies $(u_j)_{j\in\mathbb{Z}}$ of $(\varepsilon_j)_{j\in\mathbb{Z}}$ and $(w_j)_{j\in\mathbb{Z}}$ of $(\tilde{\varepsilon})_{j\in\mathbb{Z}}$ , such that 
\begin{equation}\label{eq:d3copy}
\sup_{t\in[0,1]}\Big\|\frac{1}{\sqrt{n}}\sum_{s=1}^{[nt]}\sum_{l=1}^m \I_l^\top \sum_{j=0}^\infty \Psi_j(m)\big(w_{s-j} - u_{s-j} \big) v_l\Big\|\xrightarrow{P}0.
\end{equation} 
From this, the statement of the Lemma will follow.

For the construction of the copies   $(u_j)_{j\in\mathbb{Z}}$ and $(w_j)_{j\in\mathbb{Z}}$, we use Mallows metric $d_2$. Recall the definition of this  metric according to which,   for two random vectors $ X$ and $ Y$ with $ \E\|X\|^2<\infty$ and $ \E\|Y\|^2<\infty$, $d_2(X,Y) :=\inf \big\{\E\|X-Y\|^2\big\}^{1/2}$, where the infimum is taken over all pairs of random vectors $(W,U)$ with finite second moments, such that $ {\mathcal L}(W)={\mathcal L}(X)$ and ${\mathcal L}(U)={\mathcal L}(Y)$. Here and for a random vector 
$X$, $ {\mathcal L}(X)$  denotes the law of $X$.  We refer to  Bickel and Freedman \cite{bickel1981some}, Section 8,  for more details on the $d_2$ metric and its properties.  We also write for simplicity 
$d_2(X,Y)=d_2(F_X,F_Y)$, where $ F_X$ and $F_Y$ denote the distribution functions of $ X$ and $Y$, respectively. 
Now, on  a sufficiently rich probability space, let $ (w_t, u_t)$, $t \in{\mathbb Z}$,   be  i.i.d. random vectors  satisfying $ w_t \sim \hat G_{e}^{(m)}$ and $ u_t\sim G_e^{(m)}$ and  such  that $ d_2(w_1,u_1)=\sqrt{\E\|w_1-u_1\|^2}$ holds true. Here  $ \hat G_{e}^{(m)}$ denotes the empirical distribution function of the centered $n-p$ random variables $ e_{p+1}, e_{p+2},  \ldots, e_n$.  Observe  that $ d_2(w_1, u_1)=d_2(\hat G_{e}^{(m)}, G_e^{(m)})$.  We first establish that 
\begin{equation}
 d_2(\hat G_{e}^{(m)}, G_e^{(m)}) \rightarrow 0, \ \ \mbox{in probability.}
 \end{equation}
For this we introduce some additional notation in order to  make clear the dependence of the random variables considered on $m$. In particular, we write 
$\underline{u}(m)=(u_{1}(m), u_{2}(m), \ldots, u_{m}(m))^\top$
 for the $m$-dimensional vector  having distribution function $G_e^{(m)}$ and  $\underline{w}(m)=(w_{1}(m), w_{2}(m), \ldots, w_{m}(m))^\top$
 for the $m$-dimensional vector  having distribution function $\hat G_e^{(m)}$
 Notice that for any $m\in\mathbb{N} $ it holds true that 
 \[ 0\leq \sum_{j+1}^m \E(e_{j,t}(m))^2 \leq \sum_{j=1}^m \E(\xi_{j,t}^2) = \sum_{j=1}^m \lambda_j \leq \sum_{j=1}^\infty \lambda_j=C<\infty.\]
 This implies that for any $ \epsilon>0$,  $M_\epsilon \in \mathbb{N}$ exists, such that $ \sum_{j=M_\epsilon+1}^m \E(e_{j,t}(m))^2<\epsilon $  for all $ m>M_\epsilon$.  Recall that  $ m \rightarrow \infty$ as $ n\rightarrow \infty$ and let   $n$ be   large enough such that $ m>M_\epsilon$. We then have, keeping in mind that the infimum is taken overall pairs of random vectors $(\underline{w}(m), \underline{u}(m))$ such that $\underline{u}(m)$ and $ \underline{w}(m)$ have marginal distributions $ G_e^{(m)}$ and $ \hat G_{e}^{(m)}$, respectively, that 
 \begin{align} \label{eq.d2-dec}
 d_2(\hat G_{e}^{(m)}, G_e^{(m)}) & = \inf \big\{\E\| \underline{w}(m) - \underline{u}(m)\|^2\big\}^{1/2} \nonumber \\
 & = \inf \big\{ \sum_{j=1}^{M_\epsilon}\E (w_{j}(m)-u_{j}(m))^2 + \sum_{j=M_\epsilon+1}^m \E(w_{j}(m)-u_{j}(m))^2 )\big\}^{1/2}  \nonumber \\
 & \leq   \inf \big\{ \sum_{j=1}^{M_\epsilon}\E (w_{j}(m)-u_{j}(m))^2\big\}^{1/2} + \{2\sum_{j=M_\epsilon+1}^m \E(w_{j}(m))^2\}^{1/2} \nonumber  \\ 
 & \ \ \ \ \ \  \ \   + \{2\sum_{j=M_\epsilon+1}^m\E(u_{j}(m))^2\}^{1/2}, 
 \end{align}
where the infimum in the first term of (\ref{eq.d2-dec}) is taken overall pairs $(w(M_\epsilon), u(M_\epsilon) ) $ of $M_\epsilon$-dimensional random vectors such that   $ w(M_\epsilon)\sim \hat G_{e,M_\epsilon}^{(m)}$ and $u(M_\epsilon)\sim G_{e,M_\epsilon}^{(m)}$.
Notice that the last term of (\ref{eq.d2-dec})  is smaller than $ 2\epsilon$ while the term before the last one takes with  a probability approaching one as $n\rightarrow \infty$, a value which does not exceed   $2\epsilon$. To see this, observe that
$ \E(w_j(m))^2=(n-p)^{-1}\sum_{t=p+1}^nI_j^\top (e_{t}(m)-\overline{e})(e_{t}(m)-\overline{e})^\top I_j = \E(u_j(m))^2 + O_P((n-p)^{-1})$, that is,
\begin{equation*}
\sum_{j=M_\epsilon+1}^m \E(w_{j}(m))^2  = \sum_{j=M_\epsilon+1}^m \E(e_{j}(m))^2 +\mathcal{O}_P(m/(n-p)^{-1}) \leq   \epsilon + o_P(1).
\end{equation*}
 The first term of (\ref{eq.d2-dec}) equals
$ d_2( \hat G_{e,M_\epsilon}^{(m)},G_{e,M_\epsilon}^{(m)})$, which can be bounded by 
\[ d_2( \hat G_{e,M_\epsilon}^{(m)},G_{e,M_\epsilon}^{(m)}) \leq d_2( \hat G_{e,M_\epsilon}^{(m)},G_{e,M_\epsilon}) + d_2( G_{e,M_\epsilon}^{(m)},G_{e,M_\epsilon}).\]
By Assumption~\ref{ass1}(iii),  $  G_{e,M_\epsilon}^{(m)} -G_{e,M_\epsilon} \rightarrow 0$ as $n\rightarrow \infty$ .  Furthermore,   $\hat G_{e,M_\epsilon}^{(m)} - G_{e,M_\epsilon} \rightarrow 0$, in probability.
This holds true since 
%\begin{align*}
 %|\hat G_{e,M_\epsilon}^{(m)} -G_{e,M_\epsilon}^{(m)}|  & \le |\hat G_{e,M_\epsilon}^{(m)}- G_{e,M_\epsilon}| + | G_{e,M_\epsilon}^{(m)} -G_{e,M_\epsilon}|.
 %\end{align*}
%Now, $  | G_{e,M_\epsilon}^{(m)} -G_{e,M_\epsilon}| \rightarrow 0$ by assumption,  while for the first term on the right hand side of the above inequality we have,
\begin{align*}
\big| \E(\hat G_{e,M_\epsilon}^{(m)}(x)) & -G_{e,M_\epsilon}(x)\big|  = \big|G_{e,M_\epsilon}^{(m)}(x+\overline{e}) -G_{e,M_\epsilon}(x)\big| \\
& \leq \big|G_{e,M_\epsilon}^{(m)}(x+\overline{e}) -G_{e,M_\epsilon}(x+\overline{e})\big| + \big|G_{e,M_\epsilon}(x+\overline{e}) -G_{e,M_\epsilon}(x)\big| \\
& \leq  \sup_{x} \big| G_{e,M_\epsilon}^{(m)}(x) -G_{e,M_\epsilon}(x)\big|  +  \big|G_{e,M_\epsilon}(x+\overline{e}) -G_{e,M_\epsilon}(x)\big| \rightarrow 0,
\end{align*}
 by the assumed continuity of  $G_{e,M_\epsilon}$ and the fact that $\|\overline{e}\| \rightarrow 0$, in probability.
 Also,  ${\rm Var}( \hat G_{e,M_\epsilon}^{(m)}(x)))\leq 1/(4n) \rightarrow 0$, which shows that     $\hat G_{e,M_\epsilon}^{(m)} - G_{e,M_\epsilon} \rightarrow 0$, in probability.
 For the second moments of $w(M_\epsilon) $ and $  u(M_\epsilon)  $ we have
 \[ \E w(M_\epsilon)w(M_\epsilon)^\top = E_{M_\epsilon} \Big(\frac{1}{n-p} \sum_{t=p+1}^n (e_t(m) -\bar{e} ) (e_t-(m)\bar{e})^\top \Big) E^\top_{M_\epsilon}  \]
 and
 \[ \E u(M_\epsilon)u(M_\epsilon)^\top  = E_{M_\epsilon}\Sigma_e(m) E_{M_\epsilon}^\top,\]
 where $ E_{M_\epsilon} $ is the  $M_\epsilon \times m$ matrix   $ E_{M_\epsilon} =\big(I_{M_\epsilon}, 0_{M_\epsilon \times m}\big)$ with $ I_{M_\epsilon}$ the $M_\epsilon \times M_\epsilon$ unit matrix and $ 0_{M_\epsilon \times m}$ a $ M_\epsilon \times m$ matrix of zeros. Then, 
\begin{multline*}
\| \E w(M_\epsilon)w(M_\epsilon)^\top-\E u(M_\epsilon)u(M_\epsilon)^\top\|_F \\
\leq C \big\| \frac{1}{n-p} \sum_{t=p+1}^n (e_t -\bar{e} ) (e_t-\bar{e})^\top - \Sigma_e(m)\big\|_F \stackrel{P}{\rightarrow} 0.
\end{multline*}
Therefore    by Lemma 8.3 of Bickel and Freedman  \cite{bickel1981some},   we conclude  that  
\[ d_2( \hat G_{e,M_\epsilon}^{(m)},G_{e,M_\epsilon}^{(m)})) \rightarrow 0, \ \ \mbox{in probability},\]
that means we can define copies  $(u_j)_{j\in\mathbb{Z}}$ of $(\varepsilon_j)_{j\in\mathbb{Z}}$ and $(w_j)_{j\in\mathbb{Z}}$ of $(\tilde{\varepsilon}_{j})_{j\in\mathbb{Z}}$ with $\E[\|u_1-w_1\|^2]\rightarrow 0$.

  Consider next \eqref{eq:d3copy}. For every $u\in[0,1]$, we consider the sequence
  \begin{equation*}
  \sum_{l=1}^m \I_l^\top \sum_{j=0}^\infty \Psi_j(m) (w_{s-j} -u_{s-j})v_l(u), \ \ \ s=1,...,n
  \end{equation*}
 of real-valued random variables and will apply Theorem 1 of \cite{wu2007strong}. For this, we consider the filtration $(\mathcal{F}_n)_{n\in\mathbb{N}}$ with $\mathcal{F}_n=\sigma((u_k,w_k)_{k\leq n})$ (the sigma algebra generated by $(u_k,w_k)_{k\leq n}$). Now
\begin{multline*}
\theta_{n,2}=\Big\|\E\big[\sum_{l=1}^m \I_l^\top \sum_{j=0}^\infty \Psi_j(m) (w_{n-j} -u_{n-j})v_l(u)\big|\mathcal{F}_0\big]\\
-\E\big[\sum_{l=1}^m \I_l^\top \sum_{j=0}^\infty \Psi_j(m) (w_{n-j} -u_{n-j})v_l(u)\big|\mathcal{F}_{-1}\big]\Big\|_2\\
=\Big\|\sum_{l=1}^m \I_l^\top\Psi_n(m) (w_{0} -u_{0})v_l(u)\Big\|_2\leq \sum_{l=1}^m\big\| \I_l^\top\Psi_n(m)\big\|_F \sqrt{\E\|w_{0} -u_{0}\|^2}
\end{multline*}
and consequently $\sum_{n=0}^\infty \theta_{n,2} \leq C\sqrt{\E\|w_{0} -u_{0}\|^2]}$. With Theorem 1 of \cite{wu2007strong}, we have
\begin{multline*}
\E\Big[\sup_{t\in[0,1]}\Big\|\frac{1}{\sqrt{n}}\sum_{s=1}^{[nt]}\sum_{l=1}^m \I_l\sum_{j=0}^\infty \Psi_j(m)\big(w_{s-j} - u_{s-j} \big) v_l\Big\|^2\Big]\\
\leq  \int_0^1 \E\Big[\sup_{t\in[0,1]}\Big\|\frac{1}{\sqrt{n}}\sum_{s=1}^{[nt]}\sum_{l=1}^m \I_l\sum_{j=0}^\infty \Psi_j(m)\big(w_{s-j} - u_{s-j} \big) v_l(u)\Big\|^2\Big]du\\
\leq C\E\|w_{0} -u_{0}\|^2\rightarrow 0
\end{multline*}
This completes the proof.

\vspace*{0.2cm}

\noindent{\bf Proof of Lemma~\ref{le.lemma3}:}  We write 
\begin{align*}
Z_n^\circ(t) - \tilde{Z}_{n,m}(t) = \frac{1}{\sqrt{n}} \sum_{s=1}^{\floor*{nt}} \sum_{l=1}^\infty \xi_{l,s}^\circ v_l - \frac{1}{\sqrt{n}} \sum_{s=1}^{\floor*{nt}}\sum_{l=1}^m \tilde{\xi}_{l,s} v_l = \frac{1}{\sqrt{n}} \sum_{s=1}^{\floor*{nt}} \sum_{l=m+1}^\infty \xi_{l,s}^\circ v_l
\end{align*}
Note that $\text{Var}(\xi_{l,s}^\circ)=\lambda_l  \rightarrow 0$
%\overset{l\to\infty}{\rightarrow 0}$, 
as $l \rightarrow \infty$, where $\lambda_l$ is the $l$-th largest eigenvalue of $C_0=\E[X_t \otimes X_t]$ . Using Markov's inequality we have
\begin{multline*}
\text{P}(\sup\limits_{t\in[0,1]} \Vert Z_n^\circ(t) - \tilde{Z}_{n,m}(t) \Vert > \varepsilon)  \leq \frac{1}{\epsilon^4}\E\big(\sup_{t\in[0,1]}\| \frac{1}{\sqrt{n}} \sum_{s=1}^{\floor*{nt}} \sum_{l=m+1}^\infty \xi_{l,s}^\circ v_l\|\big)^4\displaybreak[0]\\
= \frac{1}{\epsilon^4}\E\big(\max_{1\leq k\leq n}\| \frac{1}{\sqrt{n}} \sum_{s=1}^{k} \sum_{l=m+1}^\infty \xi_{l,s}^\circ v_l\|\big)^4 = \frac{1}{\epsilon^4}\E\big(\max_{1\leq k\leq n}\| \frac{1}{\sqrt{n}} \sum_{s=1}^{k} \sum_{l=m+1}^\infty \xi_{l,s}^\circ v_l\|^4\big)
\end{multline*} 
We apply Theorem~\ref{ThmB}. Using the notation $ \gamma_l(h)=Cov(\xi^\circ_{l,0},\xi^\circ_{l,h})$ and $\gamma_{l_1,l_2}(h)=Cov(\xi^\circ_{l_1,0},\xi^\circ_{l_2,h}) $, we have that 
\begin{align*}
\frac{1}{n^2} \E\Vert \sum_{s=a+1}^{a+n} \sum_{l=m+1}^\infty \xi_{l,s}^\circ v_l\Vert^4  &= \frac{1}{n^2}\sum_{s_1, \ldots, s_4=a+1}^{a+n}\sum_{l_1,\ldots, l_4=m+1}^\infty \langle v_{l_1},v_{l_2}\rangle \langle v_{l_3},v_{l_4}\rangle\\
& \ \ \ \  \times \E\big(\xi_{l_1,s_1}^\circ \xi_{l_2,s_2}^\circ  \xi_{l_3,s_3}^\circ \xi_{l_4,s_4}^\circ\big)\displaybreak[0] \\
& \leq \frac{1}{n^2} \sum_{s_1, \ldots, s_4=a+1}^{a+n}\sum_{l_1, l_2=m+1}^\infty \Big\{ | \gamma_{l_1}(s_2-s_1) || \gamma_{l_2}(s_4-s_3)| \\
& \ \ \ \ +  | \gamma_{l_1,l_2}(s_3-s_1) || \gamma_{l_1,l_2}(s_4-s_2)|  \\
& \ \ \ \  + | \gamma_{l_1,l_2}(s_4-s_1) | | \gamma_{l_1,l_2}(s_3-s_2)|\\
& \ \ \ \ + |cum(\xi^\circ_{l_1,s_1}, \xi^\circ_{l_1,s_2}, \xi^\circ_{l_2,s_3},\xi^\circ_{l_2,s_4}) | \Big\}\\
& = S_{1,n} + S_{2,n} + S_{3,n} + S_{4,n},
\end{align*}
with an obvious notation for $ S_{i,n}$, $i=1,\ldots, 4$. Using
\[ \sum_{s_1,s_2=a+1}^{a+n}  | \gamma_{l_1,l_2}(s_1-s_2)| = \sum_{h=-n+1}^{n-1} (n-|h|) |\gamma_{l_1,l_2}(h) |,\]
we get that  $ n^{-2} E\Vert \sum_{s=a+1}^{a+n} \sum_{l=m+1}^\infty \xi_{l,s}^\circ v_l\Vert^4 $ is for any $ a \in \mathbb{N} $ bounded by   
\begin{align} \label{eq.g2}
 g^2(n) := &  \Big(\sum_{l=m+1}^\infty \sum_{h=-n+1}^{n-1} (1-|h|/n)|\gamma_l(h)|\Big)^2  \nonumber \\
 & \ \ \ +  2\sum_{l_1,l_2=m+1}^\infty \Big(\sum_{h=-n+1}^{n-1} (1-|h|/n)|\gamma_{l_1,l_2}(h)|\Big)^2  \nonumber \\
 & \ \ \  + \frac{1}{n^2} \sum_{l_1,l_2=m+1}^\infty \sum_{s_4=a+1}^{a+n} \sum_{s_1,s_2,s_3=a+1-s_4}^{a+n-s_4} | cum_{l_1,l_1,l_2,l_2}(s_1,s_2,s_3)|,
 \end{align}
where $ cum_{l_1,l_1,l_2,l_2}(s_1,s_2,s_3)=cum(\xi^\circ_{l_1,0},\xi^\circ_{l_1,s_1},\xi^\circ_{l_2,s_3},\xi^\circ_{l_2,s_4})$.  Notice that  
  for any   $m\in {\mathbb N}$,   by  
\begin{equation} \label{eq.Boundg2} 
g^2(n)  \leq     \Big(\sum_{l=m+1}^\infty \sum_{h\in{\mathbb Z} }  |\gamma_l(h)|\Big)^2   +   2\sum_{l_1,l_2=m+1}^\infty \Big(\sum_{h\in {\mathbb Z}}|\gamma_{l_1,l_2}(h)|\Big)^2 + C n^{-1},
\end{equation}
where  $C = \sum_{l_1,l_2=m+1}^\infty  \sum_{s_1,s_2,s_3 \in {\mathbb Z}} | cum_{l_1,l_1,l_2,l_2}(s_1,s_2,s_3)| <\infty$ and, therefore, $g(n)$  satisfies the conditions of Theorem~\ref{ThmB}. 
By the same theorem we then have for a constant $K>0$, that 
\[ \E\big(\max_{1\leq k\leq n}\| n^{-1/2} \sum_{s=1}^{k} \sum_{l=m+1}^\infty \xi_{l,s}^\circ v_l\|^4\big)    \leq  K g^2(n) \rightarrow 0,\]
as $ n\rightarrow \infty$ because $\lim_{n\rightarrow\infty} \sum_{l=m+1}^\infty \sum_{h\in{\mathbb Z} }  |\gamma_l(h)| =0$ and\\
 $\lim_{n\rightarrow \infty} \sum_{l_1,l_2=m+1}^\infty \Big(\sum_{h\in {\mathbb Z}}|\gamma_{l_1,l_2}(h)|\Big)^2 = 0$.

\vspace*{0.2cm}

\noindent{\bf Proof of Lemma~\ref{le.lemma4}:} \ 
%For Lemma 4, a little more work is needed. \\
Recall the definition of $Z_n^\circ$:
\[Z_n^\circ(t) = \frac{1}{\sqrt{n}} \sum_{s=1}^{\floor*{nt}} \sum_{l=1}^\infty \xi_{l,s}^\circ v_l = \sum_{l=1}^\infty \frac{1}{\sqrt{n}} \sum_{s=1}^{\floor*{nt}}\xi_{l,s}^\circ v_l  \]
and define
\[Z_{n,m}^\circ(t):= \sum_{l=1}^m \frac{1}{\sqrt{n}} \sum_{s=1}^{\floor*{nt}}\xi_{l,s}^\circ v_l  \]
By Theorem 3.2 of \cite{BIL} $(Z_n^\circ(t))_{t \in [0,1]} \Rightarrow W$ holds true if we show that
\begin{itemize}
\item[I)] $(Z_{n,L}^\circ(t))_{t \in [0,1]} \Rightarrow W_L$ as $n\to\infty$ for any $L \in \mathbb{N}$ fixed. Here, $W_L$ is a Brownian Motion in $H$ with covariance operator $C_{\omega,L}$, s.t.
\[\langle C_{\omega,L}x,y\rangle = 2\pi \sum_{r=1}^L \sum_{s=1}^L f_{r,s}(0) \langle v_r,x\rangle\langle v_s,y\rangle    \]
\item[II)] $W_L \Rightarrow W$ as $L\to\infty$
\item[III)] \[\lim\limits_{L\to\infty}\limsup\limits_{n \to \infty} \Prob(\sup\limits_{t\in[0,1]} \vert Z_{n,L}^\circ(t) -Z_n^\circ(t) \vert > \varepsilon ) = 0 \;\;\; \forall \varepsilon >0. \]
\end{itemize}

\underline{I):}
%For simpler notation, recall that the first $L$  components ïf the infinite dimensional vector process  $\{\xi_s^\circ\}$,  equal  the $L$-dimensional process $\{ \tilde{\xi}_s (L) \}$. 
Recall that the first $L$ components of $\{\xi^\circ_s\}$ equal the $L$-dimensional process $\{\tilde{\xi}_s(L)\}$. Use $ \tilde{\xi}_s$ for $ \tilde{\xi}_s(L)$ in the following. Rewrite   $Z_{n,L}^\circ$ in the following way:
\begin{align*}
Z_{n,L}^\circ(t) = \sum_{l=1}^L \frac{1}{\sqrt{n}} \sum_{s=1}^{\floor*{nt}} \xi_{l,s}^\circ v_l = \sum_{l=1}^L \I_l^{\Transp}\underbrace{ \frac{1}{\sqrt{n}} \sum_{s=1}^{\floor*{nt}} \tilde{\xi}_{s}}_{=: L_n(t)} v_l = \sum_{l=1}^L\I_l^{\Transp} L_n(t) v_l
\end{align*}
with $\I_l,\; \tilde{\xi}_s,  \;L_n(t) \in \mathbb{R}^L$. Recall  that
\begin{align*}
\tilde{\xi}_s = \sum_{j=1}^\infty \Psi_j(L)\varepsilon_{s-j} + \varepsilon_s = \sum_{j=1}^\infty A_j(L) \tilde{\xi}_{s-j} + \varepsilon_s
\end{align*}
We  show that
\begin{align}
(L_n(t))_{t \in [0,1]} \Rightarrow B_L  \label{Plus}
\end{align}
where $B_L$ is a Brownian Motion in $\mathbb{R}^L$ with covariance matrix $\Gamma_L$, $\Gamma_L = 2\pi (f_{r,s}(0))_{r,s=1,...,L}$. For this, we use Theorem A.1 of \cite{AUE}. According to this theorem, (\ref{Plus}) holds true if 
\[\sum_{r\geq 1} \big(\E[\Vert \tilde{\xi}_s - \tilde{\xi}_s^{(r)} \Vert^2]\big)^{1/2} < \infty  \]
where $\tilde{\xi}_s^{(r)} = \sum_{j=1}^r \Psi_j(L) \varepsilon_{s-j}+ \varepsilon_s$, is a truncated version of $\tilde{\xi}_s$.
We have
\begin{align*}
\big( \E\Vert \tilde{\xi}_s -\tilde{\xi}_s^{(r)} \Vert ^2  \big)^{1/2} = \Big( \E\Vert \sum_{j=1}^\infty \Psi_j(L) \varepsilon_{s-j} - \sum_{j=1}^r \Psi_j(L) \varepsilon_{s-j} \Vert^2 \Big)^{1/2} \\
= \Big( \E\Vert \sum_{j=r+1}^\infty \Psi_j(L) \varepsilon_{s-j} \Vert^2\Big)^{1/2} \leq \sum_{j=r+1}^\infty \big( \E[\Vert \Psi_j(L) \varepsilon_{s-j}\Vert^2\big)^{1/2} \\
= \sum_{j=r+1}^\infty \Vert \Psi_j(L)\Vert_F \big( \E\Vert \varepsilon_{s-j}\Vert^2 \big)^{1/2} = \Vert \Sigma_\varepsilon(L)\Vert_F \sum_{j=r+1}^\infty \Vert \Psi_j(L) \Vert_F,
\end{align*} 
where $\Sigma_\varepsilon(L) = \E[\varepsilon_s \varepsilon_s^{\Transp}]$. Thus
\begin{align*}
\sum_{r \geq 1} \big( \E\Vert \tilde{\xi}_s - \tilde{\xi}_s^{(r)} \Vert^2\big)^{1/2} \leq \Vert \Sigma_\varepsilon (L) \Vert_F \sum_{r\geq 1} \sum_{j=r+1}^\infty \Vert \Psi_j(L)\Vert_F \\
\leq \Vert \Sigma_\varepsilon(L)\Vert_F \sum_{j=1}^\infty j \Vert \Psi_j(L) \Vert_F < \infty
\end{align*}
by Lemma 6.1 of \cite{PAP}. Thus we get  $(L_n(t))_{t \in[0,1]} \Rightarrow B_L$ and 
\[ (\sum_{l=1}^L \I_l^{\Transp} L_n(t) v_l) \Rightarrow W_L\]
with a Brownian motion $W_L$ that has the covariance operator $\langle W_L(x) , y \rangle = 2\pi \sum_{r=1}^L \sum_{s=1}^L f_{r,s}(0) \langle v_r,x\rangle \langle v_s,y \rangle$. \\

\underline{II):}
We have that 
\begin{align*}
& \Vert \sum_{r=1}^L \sum_{s=1}^L f_{r,s}(0) \langle v_r,x\rangle \langle v_s,y \rangle - \sum_{r=1}^\infty \sum_{s=1}^\infty f_{r,s}(0) \langle v_r,x\rangle \langle v_s,y \rangle \Vert_{HS} \\
& \leq \Vert \sum_{r=1}^L \sum_{s=1}^\infty f_{r,s}(0) \langle v_r,x\rangle \langle v_s,y \rangle - \sum_{r=1}^\infty \sum_{s=1}^\infty f_{r,s}(0) \langle v_r,x\rangle \langle v_s,y \rangle \Vert_{HS} \\
& + \Vert \sum_{r=1}^\infty \sum_{s=1}^L f_{r,s}(0) \langle v_r,x\rangle \langle v_s,y \rangle - \sum_{r=1}^\infty \sum_{s=1}^\infty f_{r,s}(0) \langle v_r,x\rangle \langle v_s,y \rangle \Vert_{HS} \\
& + \Vert \sum_{r=m+1}^\infty \sum_{s=L+1}^\infty f_{r,s}(0) \langle v_r,x\rangle \langle v_s,y \rangle - \sum_{r=1}^\infty \sum_{s=1}^\infty f_{r,s}(0) \langle v_r,x\rangle \langle v_s,y \rangle \Vert_{HS}  \overset{L\to\infty}{\rightarrow} 0;
\end{align*}
see the last step in the proof of Prop. 3.2 (\cite{PAP}).

\underline{III):} By Markov's inequality  it suffices to show that 
\[\lim\limits_{L \to\infty} \limsup\limits_{n\to\infty} \E[\big( \sup\limits_{t\in[0,1]} \Vert Z_{n,L}^\circ(t) - Z_n^\circ(t) \Vert \big)^4] = 0.\]
For this we argue as in the proof of Lemma~\ref{le.lemma3} and get the bound
\[  \E[\big( \sup\limits_{t\in[0,1]} \Vert Z_{n,L}^\circ(t) - Z_n^\circ(t) \Vert \big)^4] \leq \E\big(\max_{1\leq k\leq n}\| \frac{1}{\sqrt{n}} \sum_{s=1}^{k} \sum_{l=L+1}^\infty \xi_{l,s}^\circ v_l\|^4\big)\leq Kg^2_L(n),\]
where, similarly to   (\ref{eq.g2}), the function $g^2_L(n)$ is given here by,
\begin{align*} \label{eq.g2-2}
 g^2(n) := &  \Big(\sum_{l=L+1}^\infty \sum_{h=-n+1}^{n-1} (1-|h|/n)|\gamma_l(h)|\Big)^2  \nonumber \\
 & \ \ \ +  2\sum_{l_1,l_2=L+1}^\infty \Big(\sum_{h=-n+1}^{n-1} (1-|h|/n)|\gamma_{l_1,l_2}(h)|\Big)^2  \nonumber \\
 & \ \ \  + \frac{1}{n^2} \sum_{l_1,l_2=L+1}^\infty \sum_{s_4=a+1}^{a+n} \sum_{s_1,s_2,s_3=a+1-s_4}^{a+n-s_4} | cum_{l_1,l_1,l_2,l_2}(s_1,s_2,s_3)|,
 \end{align*}
%\[ g^2_L(n)  :=    \Big(\sum_{l=L+1}^\infty \sum_{h\in{\mathbb Z} }  |\gamma_l(h)|\Big)^2   +   2\sum_{l_1,l_2=L+1}^\infty \Big(\sum_{h\in {\mathbb Z}}|\gamma_{l_1,l_2}(h)|\Big)^2 + C n^{-1}. \]
Since
\[  \lim_{n\rightarrow\infty} g^2_L(n) =  \Big(\sum_{l=L+1}^\infty \sum_{h\in{\mathbb Z} }  |\gamma_l(h)|\Big)^2   +   2\sum_{l_1,l_2=L+1}^\infty \Big(\sum_{h\in {\mathbb Z}}|\gamma_{l_1,l_2}(h)|\Big)^2,\]
and  the limit on the  right hand side  above goes to zero as $ L\rightarrow \infty$, the proof of Lemma 5.4 is complete.

\vspace*{0.2cm}

\noindent {\bf Proof of Theorem~\ref{thm:altloc} :}
Before we start with the proof, let's introduce some notation: As only $Y_1,...,Y_n$ are observed (not $X_1,...,X_n$), estimates have to be based on this observations. To make this clear, we write
\begin{itemize}
\item $\hat{C}_{0,Y}$ for the sample covariance operator based on $Y_1,...,Y_n$
\item $\hat{v}_{j,Y}$ and $\hat{\lambda}_{j,Y}$ for its eigenvectors and eigenvalues
\item $\hat{\xi}_{t,Y}$ for score vectors with $\hat{\xi}_{j,t,Y}=\langle Y_t, \hat{v}_{j,Y}\rangle$
\item $\hat{A}_{j,m,Y}$ estimated autoregressive matrices based on $Y_1,...,Y_n$
\end{itemize} 
and so on. However, as the distribution of $Y_i$, $i=1,..,n$ is changing with $n$, we still use the original notation for true quantities related to the distribution of $X_1,...,X_n$:
\begin{itemize}
\item $C_0$: covariance operator of $X_1$
\item $v_j$ and $\lambda_j$: eigenvectors and eigenvalues of $C_0$
\item $\xi_{t,Y}$: score vectors with $\xi_{j,t,Y}=\langle Y_t,v_j
\rangle$
\item $\xi_{t}$: score vectors with $\xi_{j,t}=\langle X_t,v_j
\rangle$
\item $A_{j,m }$ autoregressive matrices for process $(\xi_{t})_{t\in\mathbb{N}}$
\end{itemize}

\vspace*{0.2cm}

\begin{lem}\label{lem:covalt} Under the assumtions of Theorem \ref{thm:altloc} , we have
\begin{equation*}
\E\left\|\hat{C}_{0,Y}-\hat{C}_{0,X}\right\|^2_{HS}=O\big(n^{-\min\{4r, 1+2r\}}\big)
\end{equation*}

\end{lem}

\begin{proof} A short calculation gives $Y_i-\bar{Y}_n=X_i-\bar{X}_n+c_{i,n}$ with $c_{i,n}=-\frac{n-k^\ast}{n^{1+r}}$ for $i\leq k^\ast$ and $c_{i,n}=-\frac{k^\ast}{n^{1+r}}$ for $i> k^\ast$. So we can conclude that
\begin{multline*}
\hat{C}_{0,Y}-\hat{C}_{0,X}=\frac{1}{n}\sum_{i=1}^n(Y_i-\bar{Y}_n)\otimes (Y_i-\bar{Y}_n)-\frac{1}{n}\sum_{i=1}^n(X_i-\bar{X}_n)\otimes (X_i-\bar{X}_n)\\
=\frac{1}{n}\sum_{i=1}^n(X_i-\bar{X}_n)\otimes c_{i,n}\mu+\frac{1}{n}\sum_{i=1}^n c_{i,n }\mu\otimes (X_i-\bar{X}_n)+\frac{1}{n}\sum_{i=1}^n c_{i,n}\mu\otimes c_{i,n}\mu
\end{multline*}
The last summand is deterministic and of order $O(n^{-2r})$, as $|c_{i,n}|\leq n^{-r}$. For the first summand, we have
\begin{equation*}
\frac{1}{n}\sum_{i=1}^n(X_i-\bar{X}_n)\otimes c_{i,n}\mu=\frac{n-k^\ast}{n^{2+r}}\Big(\sum_{i=1}^{k^\ast}X_i\Big)\otimes \mu+\frac{k^\ast}{n^{2+r}}\Big(\sum_{i=k^\ast+1}^{n}X_i\Big)\otimes \mu
\end{equation*}
so
\begin{multline*}
\E\Big\|\frac{1}{n}\sum_{i=1}^n(X_i-\bar{X}_n)\otimes c_{i,n}\mu\Big\|^2_{HS} \\
\leq \frac{C}{n^{2+2r}}\E\Big\|\sum_{i=1}^{k^\ast}X_i\Big\|^2+\frac{C}{n^{2+2r}}\E\Big\|\sum_{i=k^\ast+1}^{n}X_i\Big\|^2=O(n^{1+2c})
\end{multline*}
as $\E[\|\sum_{i=1}^{n}X_i\|^2]=O(n)$. The second summand can be treated in the same way.
\end{proof}

\noindent{\bf Proof of Theorem \ref{thm:altloc}} The statement of the theorem can be proved along the lines of Theorem \ref{th.boot}. We have to check that Lemmas  6.3, 6.5, 6.6, 6.7, 6.8 of \cite{PAP} which are used in the proof still hold.

The proof of Lemma 6.3 of \cite{PAP} is based \cite{hoerkok09}. First note that  Theorem 3.1 of \cite{hoerkok09} still holds, because by this Theorem applied to $X_1,X_2,...$ and by our Lemma \ref{lem:covalt}, we have
\begin{equation*}
\E\big\|\hat{C}_{0,Y}-C_0\big\|^2\leq 2\E\big\|\hat{C}_{0,Y}-\hat{C}_{0,X}\big\|^2+2\E\big\|\hat{C}_{0,X}-C_0\big\|^2=O\Big(\frac{1}{n}\Big).
\end{equation*}
Using Lemmas 3.1 and 3.2,  Lemma 6.3 of \cite{PAP} follows for the estimators $\hat{A}_{j,m,Y}$ the same way as before.

For Lemma 6.5 of \cite{PAP}, only parts (iii) and (iv) have to be generalized. Part (iii) does still hold because Lemma 6.3 does still hold.  For part (iv), we have to bound $\frac{1}{n-p}\sum_{t=p+1}^n\| \hat{\xi}_{t,Y}-\xi_{t}\|$ and $\frac{1}{n-p}\sum_{t=p+1}^n\| \hat{\xi}_{t-p,Y}-\xi_{t-p}\|$. We will only treat the first sum in detail
\begin{align*}
&\frac{1}{n-p}\sum_{t=p+1}^n\| \hat{\xi}_{t-p,Y}-\xi_{t-p}\|^2=\frac{1}{n-p}\sum_{t=p+1}^n\|(\langle Y_t, \hat{v}_{j,Y}\rangle-\langle X_t, v_j\rangle)_{j=1,..,m}\|^2\\
\leq &\frac{2}{n-p}\sum_{t=p+1}^n\|(\langle Y_t, \hat{v}_{j,Y}\rangle-\langle Y_t, v_j\rangle)_{j=1,..,m}\|^2\\
&\quad+\frac{2}{n-p}\sum_{t=p+1}^n\|(\langle Y_t, v_j\rangle-\langle X_t, v_j\rangle)_{j=1,..,m}\|^2\displaybreak[0]\\
=&\frac{2}{n-p}\sum_{t=p+1}^n\|(\langle Y_t, \hat{v}_{j,Y}- v_j\rangle)_{j=1,..,m}\|^2+\frac{2}{n-p}\sum_{t=p+1}^n\|(\langle Y_t- X_t, v_j\rangle)_{j=1,..,m}\|^2.
\end{align*}
For the first summand, we use the Cauchy-Schwarz inequality and obtain
\begin{multline*}
\frac{2}{n-p}\sum_{t=p+1}^n\|(\langle Y_t, \hat{v}_{j,Y}- v_j\rangle)_{j=1,..,m}\|^2\leq \frac{2}{n-p}\sum_{t=p-1}\|Y_t\|\sum_{j=1}^m\|\hat{v}_{j,Y}-v_j\|^2.
\end{multline*}
As in the proof of Lemma 6.3 of \cite{PAP}, we have $\sum_{j=1}^m\|\hat{v}_{j,Y}-v_j\|^2=\mathcal{O}_P(n^{-1}\sum_{j=1}^m \alpha_j^{-2})$. For the second summand, we use that $Y_i-X_i=n^{-r}\mu$ for $i>k^\ast$ and $Y_i-X_i=0$ otherwise, so
\begin{multline*}
\frac{2}{n-p}\sum_{t=p+1}^n\|(\langle Y_t- X_t, v_j\rangle)_{j=1,..,m}\|^2\\
=\frac{2}{n^{2r}(n-p)}\sum_{t=k^\ast+1}^n\|(\langle \mu, v_j\rangle)_{j=1,..,m}\|^2=O(n^{-2r}).
\end{multline*}
Thus $\frac{1}{n-p}\sum_{t=p+1}^n\| \hat{\xi}_{t-p,Y}-\xi_{t-p}\|^2=\mathcal{O}_P(\max\{n^{-2r},n^{-1}\sum_{j=1}^m \alpha_j^{-2})\}$. By  Assumption \ref{ass2} it holds $\frac{p}{n}\sum_{j=1}^m \alpha_j^{-2}\rightarrow 0$ and $\frac{p}{n^{2r}}\rightarrow 0$ and the rest of the proof of statement 4 of Lemma 6.5  (\cite{PAP}) works in exactly the same way as before.

Lemma 6.6 of \cite{PAP} also holds, one has to use Lemma \ref{lem:covalt} in the proof to bound $\|\hat{C}_{0,y}-C_0\|$. Lemma 6.7  of \cite{PAP} is still true for $D^\ast_{n,m}$ based on $Y_1,...,Y_n$, because we still can use Lemma 6.5. Lemma 6.8 of \cite{PAP} is also valid, because $\sum_{j=1}^m\|\hat{v}_{j,Y}-v_j\|^2=\mathcal{O}_P(n^{-1}\sum_{j=1}^m \alpha_j^{-2})$. This completes the proof.

%\end{proof}

\section{Further Simulation Results}

 \begin{figure}
 \caption{Size-corrected empirical power for under different models with a jump of $\mu$ after $100$ of $n=200$ observations  (FSB = functional sieve bootstrap, NBB = non-overlapping block bootstrap, Asymptotic = method by \cite{AUE}).}\label{fig1b}
\includegraphics[width=\textwidth]{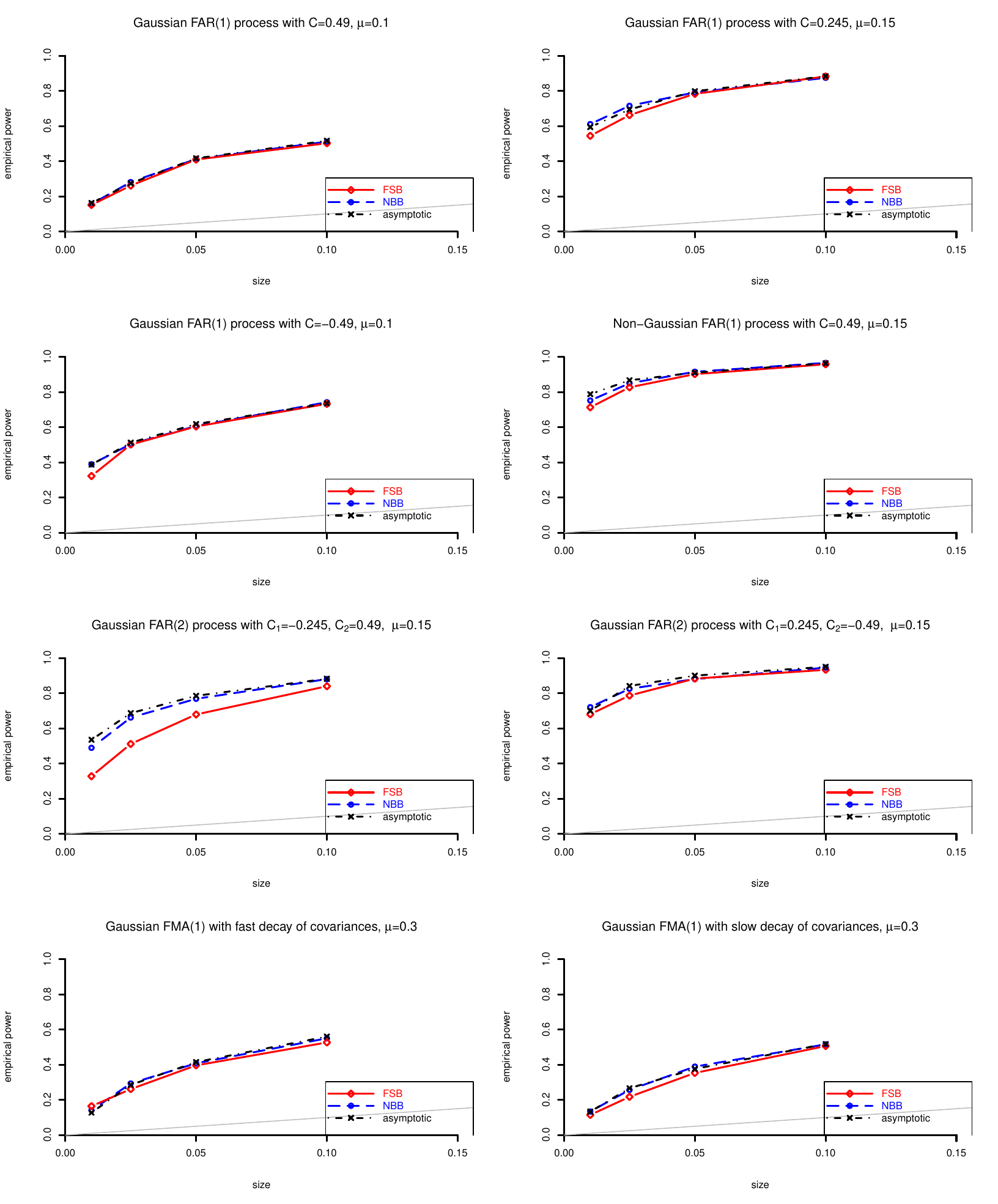} 
 \end{figure} 
 
To compare the power of the three methods (functional sieve bootstrap, non-overlapping block bootstrap, and estimation of the parameters of the limit distribution), we give addtional simulation results for a sample size of $n=200$. The observations are given by
\begin{equation*}
Y_n(t)=\begin{cases}X_n(t) \ \  &\text{for }n\leq 100\\ X_n(t)+\mu \ \ &\text{for }n\geq 101\end{cases},
\end{equation*}
 where $\mu$ is chosen to be constant (not dependent on $t$)  with values $0.1$, $0.15$ or $0.3$ depending on the dependence structure. As a stationary process $X_1,...,X_{200}$, we use FAR(1) and FMA(1) processes as described in Section \ref{sec:num}. Addtional, we simulate functional autoregressive processes of order 2 (FAR(2)) with
 \begin{equation*}
X_{n+1}(t)=C_1\int_{0}^{1}stX_n(s)ds+C_2\int_{0}^{1}stX_n(s)ds+\epsilon_{n+1}(t),
\end{equation*}
where $(\epsilon_n)_{n\in \mathbb{N}}$ are i.i.d  Brownian bridges. As can be seen in Figure \ref{fig1b}, the difference between the three methods are not very pronounced, although the over two methods have slightly higher size-correced power in most scenarios.

Addtionally, we have conducted simulations with fixed autoregressive order $p\in\{1,2,3\}$ for generating the bootstrap process. The sample size for these simulations is $n=100$ and they are based on $1000$ simulation runs. The results under null-hypothesis can be found in Table \ref{tabSize_porder}. For FAR(1)-processes, the choice $p=1$ leads to the most accurate size, while for the FMA(1)-processes, $p=3$ or $p=5$ improve the size. Under the alternative however, choosing $p=5$ reduces the power, see Table \ref{tabPower_porder}. So our recomendation is to use a low autoregressive order ($p\leq 3$) for generating the bootstrap time series. 

\begin{table}
\caption{Empirical rejection frequencies of the functional sieve bootstrap under $H_0$ for theoretical sizes $\alpha$, sample size $n=100$ and different autoregressive orders $p$ of the sieve bootstrap.}\label{tabSize_porder}
\vspace*{0.2cm}
\begin{tabular}{|l|l|ccc|}
\hline 
\rule[-1ex]{0pt}{2.5ex} model & method  & $\alpha=10\%$ & $\alpha=5\%$ & $\alpha=1\%$ \\ 
\hline 
FAR(1)  &  $p=1$ & 0.084 & 0.039 & 0.004\\ 
Gaussian  & $p=3$ & 0.078 & 0.031 & 0.001 \\ 
$C=0.49$ & $p=5$ &0.078 & 0.024 &  0.001\\ 
\hline
FAR(1)  & $p=1$ & 0.110 & 0.044 &0.008 \\ 
Gaussian  &  $p=3$ & 0.083 & 0.022 & 0.000\\ 
$C=-0.49$  & $p=5$ & 0.083& 0.028 & 0.000 \\ 
\hline
FMA(1) & $p=1$ & 0.075 & 0.026 & 0.005\\ 
fast decay &  $p=3$ & 0.086 & 0.027 & 0.000 \\ 
of autocov. & $p=5$ & 0.089 &  0.033 & 0.006 \\ 
\hline
FMA(1)&  $p=1$ &  0.077 & 0.030 & 0.003 \\ 
slow decay &  $p=3$ & 0.087 & 0.032 & 0.001 \\ 
of autocov. & $p=5$ & 0.091 & 0.032 & 0.004\\ 
\hline 
\end{tabular} 
\end{table}

\begin{table}
\caption{Empirical rejection frequencies of the functional sieve bootstrap under $H_1$ for theoretical sizes $\alpha$, sample size $n=100$ with jump size $\mu=0.15$ (FAR(1) models) or $\mu=0.3$ (FMA(1) models) after 50 observations and different autoregressive orders $p$ of the sieve bootstrap.}\label{tabPower_porder}
\vspace*{0.2cm}
\begin{tabular}{|l|l|ccc|}
\hline 
\rule[-1ex]{0pt}{2.5ex} model & method  & $\alpha=10\%$ & $\alpha=5\%$ & $\alpha=1\%$ \\ 
\hline 
FAR(1)  &  $p=1$ & 0.498 & 0.348 & 0.090\\ 
Gaussian  & $p=3$ & 0.457 & 0.228 & 0.025 \\ 
$C=0.49$ & $p=5$ & 0.384& 0.187 &  0.010\\ 
\hline
FAR(1)  & $p=1$ &  0.786 &  0.662 & 0.352\\ 
Gaussian  &  $p=3$ & 0.712 &  0.504& 0.136\\
$C=-0.49$  & $p=5$ & 0.658 & 0.376 & 0.040 \\ 
\hline
FMA(1) & $p=1$ & 0.313 & 0.184 & 0.044 \\ 
fast decay &  $p=3$ & 0.317 & 0.192 & 0.046\\ 
of autocov. & $p=5$ & 0.283 & 0.131 & 0.015 \\ 
\hline
FMA(1)&  $p=1$ & 0.327 & 0.194 & 0.044 \\ 
slow decay &  $p=3$ & 0.299 & 0.171 & 0.026 \\ 
of autocov. & $p=5$ & 0.235 & 0.093 & 0.007\\ 
\hline 
\end{tabular} 
\end{table}

\end{document}